\documentclass{amsart}
\usepackage[utf8]{inputenc}
\usepackage[margin=1.3in]{geometry}
\usepackage{amsmath,amsthm,amsfonts,amssymb}
\usepackage[all]{xy}
\usepackage{color}
\usepackage{xparse}
\usepackage{verbatim}
\usepackage{aliascnt}
\usepackage[colorlinks,citecolor=blue,linkcolor=blue,urlcolor=blue,filecolor=blue]{hyperref}
\usepackage{bbm}
\usepackage{enumerate}
\usepackage{tikz-cd}

\theoremstyle{definition}
\newtheorem{definition}{Definition}[section]
\newtheorem{example}[definition]{Example}
\newtheorem{remark}[definition]{Remark}

\theoremstyle{plain}
\newtheorem{lemma}[definition]{Lemma}
\newtheorem{theorem}[definition]{Theorem}

\newtheorem{cor}[definition]{Corollary}
\newtheorem{prop}[definition]{Proposition}

\newtheorem*{theorem*}{Theorem}
\newtheorem*{prop*}{Proposition}
\newtheorem*{cor*}{Corollary}

\newtheorem{introthm}{Theorem}

\newcommand{\thh}{\mathrm{THH}}
\newcommand{\Map}{\mathrm{Map}}

\newcommand{\colim}{\mathrm{colim}}
\newcommand{\Sp}{\mathrm{Sp}}
\newcommand{\Mod}{\mathrm{Mod}}
\newcommand{\CycSp}{\mathrm{CycSp}}
\newcommand{\stMod}{\mathrm{stMod}}

\DeclareMathOperator{\psl}{PSL}
\DeclareMathOperator{\pgl}{PGL}
\DeclareMathOperator{\spl}{SL}
\DeclareMathOperator{\gl}{GL}
\DeclareMathOperator{\fun}{Fun}

\newcommand{\tc}{\mathrm{TC}}
\newcommand{\cyclicorbit}{\mathcal{O}_{\mathcal{F}_\mathrm{cyc}}}

\usepackage[mathscr]{euscript}

\title{The algebraic K-theory of $k[\spl_2(\mathbb{F}_q)]$}
\author{Isaac Moselle}
\address{Department of Mathematical Sciences, University of Copenhagen, Denmark}
\email{ismo@math.ku.dk}


\begin{document}
\begin{abstract}
    We compute via trace methods the higher algebraic $K$-theory of the group ring $k[\spl_2(\mathbb{F}_q)]$, as well as the group rings of the related groups  $\psl_2(\mathbb{F}_q)$, $\pgl_2(\mathbb{F}_q)$, and $\gl_2(\mathbb{F}_q)$, where $k$ is a finite field of characteristic $p$ and $q=p^r$. At the core of the computation is the algebraic $K$-theory of the group ring of the Sylow $p$-subgroup, $k[C_p^r]$, which we determine via a theorem of L\"uck--Reich--Rognes--Varisco on cyclic assembly for topological cyclic homology. In the process, we reprove the cyclic assembly result in the language of Nikolaus--Scholze, analyse assembly for smaller families of subgroups, and develop further tools for computing topological cyclic homology of group rings.
\end{abstract}

\maketitle

\section*{Introduction}

Modular representation theory studies modules over group rings $k[G]$, where $k$ is a field of characteristic $p>0$ and $G$ is a finite group. When $p$ divides the order of $G$, the resulting category exhibits a rich structure related to the $p$-local subgroup structure of $G$. This raises the natural question of whether the algebraic $K$-theory of $k[G]$ admits explicit computation. The groups $K_i(k[G])$ are well understood for $i=0,1$. By work of Brauer, the Grothendieck group $K_0(k[G])$ is free of rank equal to the number of $k$-conjugacy classes of $p'$-elements in $G$, and plays a fundamental role in the theory of Brauer characters. The group $K_1(k[G])$ is a certain quotient of the abelianisation of the group of units of $k[G]$ (\cite[Lemma III.1.4]{weibel2013k}), investigated in \cite{magurn2006explicit, magurn2007explicit}. For higher $i$, however, very few explicit computations of higher $K$-groups of modular group rings are known (although see \cite{vogeli2025derived} for results when the Sylow $p$-subgroup is cyclic of order $p$).

In this paper, we compute the $p$-torsion in the higher algebraic $K$-theory of $k[G]$, for $G$ one of the groups
\[
\psl_2(\mathbb{F}_q), \quad \pgl_2(\mathbb{F}_q), \quad \spl_2(\mathbb{F}_q), \quad \gl_2(\mathbb{F}_q),
\]
and $k$ a finite field of characteristic $p$ with $q=p^r$. These groups play a foundational role in finite group theory and provide a natural testing ground for explicit calculations. From this computation, we obtain the following description of $K_i(k[G])$.

\begin{introthm}\label{intro-thm-a}
    Suppose that $G$ is one of the groups above and $k$ is a finite field of characteristic $p$, where $q=p^r$. Then for $i>0$
    \[ K_i(k[G])\cong k^{n_i} \oplus K_i(k[G]/J) \] 
   where $J$ is the Jacobson radical and $k[G]/J$ the semisimple quotient. The value $n_i$ is given by the coefficient of $x^i$ in the Hilbert series
   \[ C(1+x+x^2+\dots)^{r-1}(x+2x^3+3x^5+\dots) \]
    where the constant $C$ has values
    \begin{itemize}
        \renewcommand\labelitemi{--}
        \item $C=2$ for $G=\psl_2(\mathbb{F}_q)$ with $q$ odd,
        \item $C=1$ for $G=\pgl_2(\mathbb{F}_q)$,
        \item $C=4$ for $G=\spl_2(\mathbb{F}_q)$ with $q$ odd,
        \item $C=q-1$ for $G=\gl_2(\mathbb{F}_q)$.
    \end{itemize}
    For $q$ even, $\psl_2(\mathbb{F}_q)=\spl_2(\mathbb{F}_q)\cong\pgl_2(\mathbb{F}_q)$, so that $C=1$ in these cases.
\end{introthm}

The ring $k[G]/J$ is Morita equivalent to a product of finite fields (see Lemma \ref{artin-wed}). It follows from Quillen's computation (\cite[Theorem 8]{quillen1972cohomology}) that $K_i(k[G]/J)$ is a finite group of order prime to $p$ for $i>0$, so Theorem \ref{intro-thm-a} produces a complete description of $K_i(k[G])$. Many of our intermediate results hold over an arbitrary perfect field of characteristic $p$, but we restrict to finite fields in the main theorems. Since $K_i(k[G]/J)$ only relies on the semisimple quotient, it is the $p$-torsion component that is of primary interest from the perspective of group theory.

Our calculation is entirely based on the framework of trace methods. Trace methods approximate algebraic $K$-theory via invariants arising from Hochschild homology that are more amenable to computation---in particular, topological cyclic homology ($\tc(-;p)$). This approach is very effective for finite dimensional algebras over perfect fields. Indeed, for such an algebra $A$ and $i>0$, the map $K_i(A) \xrightarrow{} \tc_i(A;p)$ exhibits $\tc_i(A;p)$ as the $p$-torsion component of $K_i(A)$. The $\tc$ calculations have been performed for truncated polynomials in a single variable (\cite{hesselholt1997cyclic, speirs2020k}), some cases of truncated polynomials in many variables (\cite{angeltveit2014algebraic}), and for coordinate-axis algebras (\cite{hesselholt2007k, speirs2021k}). 

We proceed, then, by computing the topological cyclic homology of $k[G]$. Our approach, mirroring techniques in group cohomology, is to first reduce the problem to the Sylow $p$-subgroup $C_p^r$. Accordingly, we examine the topological cyclic homology of the elementary abelian group ring $k[C_p^r]$, determining these groups alongside the action induced by certain normalisers. Because $k[C_p^r] \cong k[x_1,\ldots,x_r]/(x_1^p,\ldots,x_r^p)$ in characteristic $p$, this calculation is equivalent to computing $\tc$ of a truncated polynomial algebra where all exponents are $p$. From this perspective, our explicit determination complements \cite[Theorem 1.3]{angeltveit2014algebraic} via entirely different methods, and is therefore of independent interest.

\begin{introthm}\label{intro-thm-b}
    For $k$ a finite field of characteristic $p$, we have for $i>0$
    \[ \tc_i(k[C_p^r];p) \cong k^{n_i} \]
    where the $n_i$ are the coefficients of $x^i$ in the Hilbert series
    \[ (p^r-1)(1+x+x^2+\dots)^{r-1}(x+2x^3+3x^5+\dots). \]
\end{introthm}

To obtain our results, we develop tools for computing the topological cyclic homology of group rings. The topological Hochschild homology of group rings is intimately related to structures on the free loop space $LBG=\mathrm{Map}(S^1,BG)$, and we analyze $LBG$ to derive results on the topological cyclic homology. Since free loops are controlled by conjugacy classes of cyclic subgroups, this naturally leads to assembly phenomena over cyclic subgroups.

While our primary results only require coefficients in $k$, in the first two sections we work in the generality of coefficients in a connective ring spectrum $R$.

\subsection*{Cyclic assembly}
For the calculation in Theorem \ref{intro-thm-b}, we use a theorem of L\"uck--Reich--Rognes--Varisco (\cite[Theorem 1.2]{luck2019assembly}). Informally, the theorem states that for a group $G$ (with a certain finiteness hypothesis) and a connective ring spectrum $R$, the spectrum $\tc(R[G];p)$ may be recovered from the spectra $\tc(R[H];p)$ as $H$ ranges over the cyclic subgroups. Formally, if we denote by $\mathcal{O}_{\mathcal{F}_\mathrm{cyc}}(G)$ the category of transitive $G$-sets with cyclic isotropy, one can construct a functor $\mathcal{O}_{\mathcal{F}_\mathrm{cyc}}(G) \to \Sp$ assigning the spectrum $\tc(R[H];p)$ to $G/H$, and the corresponding assembly map
\[ \underset{\mathcal{O}_{\mathcal{F}_\mathrm{cyc}}(G)}{\colim} \ \tc(R[H];p)\to\tc(R[G];p) \]
is an equivalence. One can view this result as analogous to Artin's induction theorem, with the caveat that the associated spectral sequence in no way collapses. In this paper, we reprove the cyclic assembly theorem for $\tc$ in the language of Nikolaus--Scholze, and show that it holds for all (possibly infinite) groups after $p$-completion (Corollary \ref{cyc-induct-tc}). We furthermore analyze the assembly map for smaller families of subgroups, in particular the family of cyclic subgroups of order prime to $p$ (see Corollary \ref{tc-cyc-prime}).

For an elementary abelian group, the structure of the category $\mathcal{O}_{\mathcal{F}_\mathrm{cyc}}(G)$ is particularly simple, making the colimit amenable to computation. Moreover, the non-trivial cyclic subgroups are all isomorphic to $C_p$, and the $\tc$ of the ring $k[C_p]\cong k[x]/(x^p)$ is well understood (\cite{speirs2020k}). Hence, we use cyclic assembly to give the description of $ \tc_i(k[C_p^r];p)$ in Theorem \ref{intro-thm-b}.

\subsection*{Frobenius groups}
To obtain Theorem \ref{intro-thm-a} from Theorem \ref{intro-thm-b}, we step up from the Sylow subgroup to its normaliser. For the groups considered here, this normaliser is the Borel subgroup of upper triangular matrices. In the cases of $\psl_2(\mathbb{F}_q)$ and $\pgl_2(\mathbb{F}_q)$, these normalisers have a straightforward structure as Frobenius groups. Recall that a semi-direct product $K \rtimes H$ is a Frobenius group if $H \cap H^g$ is trivial for all $g \in (K \rtimes H) \setminus H$. Just as the complex representation theory of such a group cleanly decomposes into representations of $H$ (via inflation) and $K$ (via induction), we show that $\tc(R[G])$ admits a similar decomposition. 

\begin{prop*}
    Suppose $G\cong K \rtimes H$ is a Frobenius group. For $R \in \mathrm{CAlg}(\Sp)$ connective, the diagram
    % https://q.uiver.app/#q=WzAsNCxbMCwwLCJcXHRjKFIpX3toSH0iXSxbMiwwLCJcXHRjKFJbSF0pIl0sWzAsMiwiXFx0YyhSW0tdKV97aEh9Il0sWzIsMiwiXFx0YyhSW0ddKSJdLFswLDJdLFsyLDNdLFswLDFdLFsxLDNdXQ==
\[\begin{tikzcd}
	{\tc(R;p)_{hH}} && {\tc(R[H];p)} \\
	\\
	{\tc(R[K];p)_{hH}} && {\tc(R[G];p)}
	\arrow[from=1-1, to=1-3]
	\arrow[from=1-1, to=3-1]
	\arrow[from=1-3, to=3-3]
	\arrow[from=3-1, to=3-3]
\end{tikzcd}\]

    becomes a pushout after $p$-completion. If moreover the order of $H$ is prime to $p$, then the map $\tc(R[K];p)_{hH}\to \tc(R[G];p)$ is split after $p$-completion.
\end{prop*}

As an example, consider the group $\mathbb{F}_q \rtimes \mathbb{F}_q^\times$, which is the normaliser of a Sylow $p$-subgroup in $\pgl_2(\mathbb{F}_q)$. The above proposition implies that for any connective $R$, we have 
\[ \tc_i(R[\mathbb{F}_q \rtimes \mathbb{F}_q^\times];p) \cong \tc_i(R[\mathbb{F}_q];p)_{\mathbb{F}_q^\times} \oplus \operatorname{coker}\big(\tc_i(R;p) \to \tc_i(R[\mathbb{F}_q^\times];p)\big) \]
up to $p$-completion. As part of our calculation of Theorem \ref{intro-thm-b}, we compute the action of $\mathbb{F}_q^\times$ on $\tc_i(k[C_p^r];p)$. Using this tool, alongside further analysis for the families $\spl_2(\mathbb{F}_q)$ and $\gl_2(\mathbb{F}_q)$, we calculate the groups $\tc_i(k[N];p)$, where $N$ is the normaliser of a Sylow $p$-subgroup.

\subsection*{Reduction to $p$-local subgroups}
The final step in proving Theorem \ref{intro-thm-a} is passing from the Sylow subgroup to the entire group $G$. As we show, for the families listed, this follows immediately from an analysis of their stable module categories. However, a secondary goal of this paper is to demonstrate how computing $\tc_i(k[G];p)$ for a general finite group $G$ may be reduced entirely to computations involving $p$-local subgroups (subgroups containing a non-trivial normal $p$-subgroup).

To this end, we describe in Appendix A induction theory for the genuine $G$-spectrum $\tc(k[-];p)$, showing that it is projective relative to the hyperelementary subgroups (and relative to the elementary subgroups when $k$ contains enough roots of unity). While much of this induction theory---studied previously by Vogeli in \cite{vogeli2025derived} and used to great effect---will be known to experts, we provide a unified reference here. We deduce the following result, which allows one to study the groups $\tc_i(k[G];p)$ via the Brown complex of non-trivial $p$-subgroups (which is homotopy equivalent to the Quillen complex of non-trivial elementary abelian $p$-subgroups). Recall that this is the simplicial complex $\Delta_p(G)$ with $n$-simplices given by chains $P_0 < \ldots <P_n$ of non-trivial $p$-subgroups. It has the property that the stabiliser $G_\sigma$ of any simplex is a $p$-local subgroup of $G$.

\begin{prop*}
    Let $k$ be a perfect field of characteristic $p$ containing a primitive $e^\mathrm{th}$ root of unity, where $e$ is the $p'$-part of the exponent of $G$. For each $i>0$, there is a split exact sequence
    \begin{align*}
    0 \to \tc_i(k[G];p) \to \bigoplus_{\sigma \in (\Delta_p(G))_0/G} \tc_i(k[G_\sigma];p)
    &\to \bigoplus_{\sigma \in (\Delta_p(G))_1/G} \tc_i(k[G_\sigma];p) \to \dots
    \end{align*}    
    where $(\Delta_p(G))_n$ denotes the set of $n$-simplices of $\Delta_p(G)$.
\end{prop*}

\subsection*{Acknowledgments}
I would like to thank my supervisor Jesper Grodal for guidance on the content of this paper, as well as many encouraging conversations. I also thank Maxime Ramzi for several helpful comments and corrections. I would like to thank Emma Brink and Jonathan Clivio for useful comments on the introduction, and Chase Vogeli for helpful discussions. I would also like to thank Christian Kremer for pointing out several errors in the first version of this paper.

I was supported by the Danish National Research
Foundation through the Copenhagen Centre for Geometry and Topology (DNRF151).

\subsection*{Outline}
In Section 1, we recall material relating to free loop spaces. We analyze the free loop space on $BG$ and state the description of the cyclotomic spectrum $\thh(k[G])$ in terms of the free loop space. 

In Section 2, we use the analysis of Section 1 to deduce results on $\thh$ and $\tc$ of group rings. In particular, we prove that cyclic assembly holds for $\tc$, and examine smaller families of subgroups.

In Section 3, we calculate $\tc(k[C_p^r];p)$ by combining cyclic assembly with the usual assembly sequence. We also recall material relating to the $r=1$ case.

In Section 4, we apply the results of earlier sections to achieve the main computation of the paper, Theorem \ref{main-thm}. We compute the topological cyclic homology of the group rings, and finally recall how $\tc$ relates to the $K$-theory groups.

Finally, in Appendix A, we review the theory of spectral Mackey functors and show how the $p$-local reductions used in Section 4 can be framed in terms of Brown's simplicial complex.

\subsection*{Conventions}
All categories are $\infty$-categories, in the sense of \cite{lurie2009higher, lurie2012}. We denote by $\mathcal{S}$ the category of spaces, and $\mathbb{S}[-]$ the suspension spectrum functor.

Throughout the paper, we use the language of cyclotomic spectra developed by Nikolaus--Scholze, \cite{nikolaus2018topological}. We will use $\tc(-;p)$ to refer to the $p$-typical variant of topological cyclic homology.
\section{Free loop spaces}
In this section, we recall some material relating to free loop spaces. Free loop spaces carry an action of a certain monoid (the Witt monoid), which can be used to describe the cyclotomic structures on Hochschild homology. We analyse this action in the case of the free loop space on $BG$ for $G$ a discrete group, and show how it relates to the well known centraliser decomposition. Finally we prove an assembly theorem for the free loop space, which leads to assembly results for Hochschild and cyclic homology of group rings.

\subsection{The Witt monoid and Frobenius lifts}
For a space $X \in \mathcal{S}$, the free loop space $LX:= \Map(S^1,X)$ has a rich structure. Firstly, there is an $S^1$-action coming from the action of $S^1$ on itself. Secondly, for each $n \in \mathbb{N}$ there is an $S^1$-equivariant equivalence
\[ \Map(S^1,X) \xrightarrow{\simeq}\Map(S^1_{hC_n}, X) \simeq \Map(S^1,X)^{hC_n} \]
where $\Map(S^1,X)^{hC_n}$ carries the residual $S^1/C_n \simeq S^1$ action (informally, the equivalence sends a loop in $X$ to the same loop repeated $n$ times). In particular, we have an $S^1$-equivariant map $LX \xrightarrow{} LX^{hC_n}$ for each $n \in \mathbb{N}$, as well as various compatibilities between them. Such maps are known as Frobenius lifts (see for example \cite{antieau2021cartier}). Following \cite{mccandless2021curves}, the structure of Frobenius lifts may be formalised as the action of a certain monoid (the Witt monoid). We include also the $p$-typical version of this construction, which we will utilize later.

\begin{definition}\cite[Construction 2.1.1]{mccandless2021curves}
    The monoid $\mathbb{N}^\times$ acts on $S^1$ via $z \mapsto{} z^k$ for $k \in \mathbb{N}^\times$. Let the Witt monoid $\mathbb{W}$ be the semi-direct product
    \[ \mathbb{W} = S^1 \rtimes \mathbb{N}^\times \]
    where $\mathbb{N}^\times$ is regarded as a discrete topological monoid.

    For a prime $p$, let $\mathbb{W}_p \subseteq \mathbb{W}$ be the submonoid $S^1 \rtimes \mu_p$, where $\mu_p=\{1,p,p^2,\ldots\} \subseteq \mathbb{N}^\times$.
\end{definition}

For any $\mathcal{C}$, we can now define the category of objects in $\mathcal{C}$ with Frobenius lifts as a presheaf category on $B\mathbb{W}$ - equivalently, the category of objects in $\mathcal{C}$ with a (right) action of $\mathbb{W}$.

\begin{definition}\cite[Definition 2.12]{mccandless2021curves}
    For a category $\mathcal{C}$, the category of objects in $\mathcal{C}$ with Frobenius lifts is
    \[ \mathcal{C}^\mathrm{Fr} = \fun(B\mathbb{W}^\mathrm{op}, \mathcal{C}) = \mathcal{P}_\mathcal{C}(B\mathbb{W}). \]
    Similarly, the category of objects with a $p$-Frobenius lift is
    \[ \mathcal{C}^{p\mathrm{Fr}} = \fun(B\mathbb{W}_p^\mathrm{op}, \mathcal{C}) = \mathcal{P}_\mathcal{C}(B\mathbb{W}_p). \]
    When $\mathcal{C}$ is symmetric monoidal, we will equip $\mathcal{C}^\mathrm{Fr}$ (and $\mathcal{C}^{p\mathrm{Fr}}$) with the pointwise symmetric monoidal structure.
\end{definition}

By restricting along $S^1 \xrightarrow{} \mathbb{W}$, we obtain a functor
\[ \mathcal{C}^\mathrm{Fr} \xrightarrow{} \fun((BS^1)^\mathrm{op}, \mathcal{C}) \simeq \fun(BS^1, \mathcal{C}) \]
that sends an object with Frobenius lifts to the underlying object with $S^1$-action. The extra structure coming from $\mathbb{N}^\times$ is precisely the Frobenius lifts $X \xrightarrow{} X^{hC_n}$ for each $n$, along with compatibilities between them. Explicitly, an element $n \in \mathbb{N}^\times$ acts via a map $X \xrightarrow{} X$, and the structure of the Witt monoid ensures that this factors through the fixed points of $C_n \subseteq S^1$. By contrast, the data of an object with a $p$-Frobenius lift is equivalent to an object $X$ with an $S^1$-action along with a single map $X \xrightarrow{} X^{hC_p}$ (this follows from $\mu_p \cong \mathbb{N}$ being the free $\mathbb{E}_1$-monoid). There is an evident forgetful functor $\mathcal{C}^\mathrm{Fr} \xrightarrow{}\mathcal{C}^{p\mathrm{Fr}}$ which we will often use implicitly.

Now, there is a homomorphism $\mathbb{W} \xrightarrow{} \mathrm{End}(S^1)$ given by the action of $\mathbb{W}$ on $S^1$,
\[ (\theta, n) \cdot z = \theta z^n. \]
The free loop space $LX$ carries a right action of $\mathrm{End}(S^1)$. We use this to define an action of $\mathbb{W}$ on $LX$. We will also refer to the resulting functor $\mathcal{S} \xrightarrow{} \mathcal{S}^\mathrm{Fr}$ as $L(-)$.

\begin{definition}
    Let $L(-): \mathcal{S} \xrightarrow{} \mathcal{S}^\mathrm{Fr}$ be the composite
    \[\mathcal{S} \xrightarrow{} \mathcal{P}_\mathcal{S}(B\mathrm{End}(S^1)) \xrightarrow{} \mathcal{P}_\mathcal{S}(B\mathbb{W}) = \mathcal{S}^\mathrm{Fr} \]
    where the first functor is the restricted Yoneda embedding, and the second is restriction along $B\mathbb{W} \xrightarrow{} B\mathrm{End}(S^1)$.
\end{definition}

For any functor $F: \mathcal{C} \xrightarrow{} \mathcal{D}$, there is a functor $F^\mathrm{Fr}: \mathcal{C}^\mathrm{Fr} \xrightarrow{} \mathcal{D}^\mathrm{Fr}$. In particular, for $R \in  \Sp$, we obtain the composite
\[ \mathcal{S} \xrightarrow{L(-)} \mathcal{S}^\mathrm{Fr} \xrightarrow{\mathbb{S}[-]} \Sp^\mathrm{Fr} \xrightarrow{R \otimes -} \Mod(R)^\mathrm{Fr} \]
sending a space $X$ to the $R$-homology of the free loop space $R[LX]$, so that $R[LX]$ has an action of the Witt monoid functorial in maps of spaces.  

\subsection{The free loop space on $BG$}

Let $G$ be a discrete group. There is a well-known decomposition of the free loop space $LBG$ in terms of centraliser subgroups. Fix a set of representatives for the conjugacy classes in $G$. Then
\[ LBG \simeq \bigsqcup_{[g]} BC_G(g) \]
where $C_G(g)$ denotes the centraliser. We explain the centraliser decomposition, and how it relates to the structure of Frobenius lifts, in the next proposition.

\begin{prop}[Centraliser decomposition]\label{cent-decomp}
    The components of the space $LBG$ are indexed by the conjugacy classes of $G$, and there is an $S^1$-equivariant decomposition
    \[  LBG \simeq \bigsqcup_{[g]} BC_G(g).  \]
    Moreover, the action of $n \in \mathbb{N}^\times \subseteq \mathbb{W}^\mathrm{op}$ on $\pi_0(LBG)$ sends the component indexed by $[g]$ to the component indexed by $[g^n]$. 
\end{prop}
\begin{proof}
    The space $BG$ is modeled by the groupoid of free transitive $G$-sets (there is a unique such set up to non-canonical choice). Hence $LBG$ is the groupoid of pairs $(X, \phi)$ with $X$ a transitive $G$-set, $\phi$ an automorphism of $X$. By construction, $n \in \mathbb{N}^\times$ sends the object $(X, \phi)$ to $(X, \phi^n)$. Every automorphism of the $G$-set $G/e$ is given by right multiplication by some element in $G$, denoted by $r_g$. Hence there is a non-canonical isomorphism $(X, \phi) \cong (G/e, r_g)$ for every $(X, \phi) \in LBG$. Now $(G/e, r_g) \cong (G/e, r_h)$ if and only if there is some $r_k$ such that 
    % https://q.uiver.app/#q=WzAsNCxbMCwwLCJHL2UiXSxbMSwwLCJHL2UiXSxbMCwxLCJHL2UiXSxbMSwxLCJHL2UiXSxbMCwxLCJyX2siXSxbMCwyLCJyX2ciLDJdLFsxLDMsInJfaCJdLFsyLDMsInJfayIsMl1d
\[\begin{tikzcd}
	{G/e} & {G/e} \\
	{G/e} & {G/e}
	\arrow["{r_k}", from=1-1, to=1-2]
	\arrow["{r_g}"', from=1-1, to=2-1]
	\arrow["{r_h}", from=1-2, to=2-2]
	\arrow["{r_k}"', from=2-1, to=2-2]
\end{tikzcd}\]
    commutes, which holds if and only if there is some $k \in G$ such that $h=kgk^{-1}$. Similarly, the automorphism group of $(G/e, r_g)$ is precisely the group $C_G(k)$, proving the first claim. By definition, $n \in \mathbb{N}^\times$ sends $(G/e, r_g)$ to $(G/e, r_{g^n})$, which proves the second claim. 
\end{proof}

We see that the structure of Frobenius lifts on $LBG$ is a homotopy theoretic version of the exponent maps $[g] \mapsto{} [g^n]$ on the set of conjugacy classes in $G$. We now want to construct the structure of Frobenius lifts on subspaces of $LBG$. This will correspond to collections of conjugacy classes closed under $[g] \xrightarrow{} [g^n]$.

\begin{definition}\label{LTBG-def}
    For $T \subseteq G$ a subset closed under conjugation, let $L_TBG \subseteq LBG$ be the subspace of components indexed by conjugacy classes lying in $T$. 
\end{definition}

Hence by definition, upon picking representatives for conjugacy classes in $T$ we have the centraliser decomposition
\[ L_TBG \simeq \bigsqcup_{[g] \in T} BC_G(g). \]

\begin{prop}\label{sub-frob-lift}
    Suppose that $T$ is closed under conjugation and $g \mapsto g^n$ for all $n > 0$ (resp. closed under conjugation and $g \mapsto g^p$). Then the action of $\mathbb{W}^\mathrm{op}$ (resp., $\mathbb{W}_p^\mathrm{op}$) on $LBG$ restricts to $L_TBG$, so that $L_TBG$ is a space with Frobenius lifts (resp., a space with a $p$-Frobenius lift).
\end{prop}
\begin{proof}
    The inclusion $L_TBG \subseteq LBG$ is an inclusion of components, so it suffices to see that the action of $\mathbb{W}$ (resp $\mathbb{W}_p$) on $\pi_0(LBG)$ restricts to an action on $\pi_0(L_TBG)$. But this is immediate from Proposition \ref{cent-decomp} (note that $S^1$ necessarily acts trivially on $\pi_0(LBG)$).
\end{proof}

\begin{remark}
    Observe that the structure of Frobenius lifts does not carry the information corresponding to the inversion $[g] \mapsto{} [g^{-1}]$. This is because we only remember the action by $\mathbb{W} \subseteq \mathrm{End}(S^1)$, not the entire monoid $\mathrm{End}(S^1)\simeq S^1 \rtimes (\mathbb{Z},\times)$.
\end{remark}

\begin{example}
    Let $N \leq G$ be a normal subgroup of $G$. Then $N$ is closed under conjugation and $g \mapsto{} g^n$, so $L_NBG$ is a space with Frobenius lifts.
\end{example}
\begin{example}\label{ex-lbg-decomp}
    Let $G=S_3$ and $p=3$. There are three conjugacy classes ($[1]$, $[(123)]$, $[(12)]$), and the subset $[(12)]$ is closed under $g \mapsto g^3$. Thus we obtain a decomposition
    \[ LBS_3 \simeq L_{[1]\cup[(123)]}BS_3 \sqcup  L_{[(12)]}BS_3\]
    as spaces with a $p$-Frobenius lift. From the centraliser decomposition, we can see that the inclusion $C_2 \xrightarrow{} S_3$ induces an equivalence $L_{\{g\}}BC_2 \xrightarrow{\simeq} L_{[(12)]}BS_3$ (where $g$ is the non-trivial element of $C_2$).
\end{example}

\subsection{Assembly theorems for $LBG$}
We now prove an assembly theorem for the free loop space. In order to make the statement self contained, we introduce a few concepts from equivariant homotopy theory.

\begin{definition}
    Let $\mathcal{O}(G)$ (the orbit category of $G$) be the category of transitive $G$-sets. For $\mathcal{F}$ a family of subgroups closed under conjugation and subgroups, let $\mathcal{O}_{\mathcal{F}}(G)$ be the full subcategory of $G$-sets with isotropy in $\mathcal{F}$. We will denote the family of cyclic subgroups by $\mathcal{F}_{\text{cyc}}$.
\end{definition}

There is a functor from $\mathcal{O}(G)$ to spaces given by $X \mapsto X_{hG}$. This sends the $G$-set $G/H$ to $BH$. Note that although $BH$ is naturally pointed, this functor does not preserve basepoints. For any $\mathcal{C}$ we obtain via restriction functors
\[ \text{Fun}(\mathcal{S}, \mathcal{C}) \xrightarrow{} \text{Fun}(\mathcal{O}(G), \mathcal{C}) \xrightarrow{} \text{Fun}(\mathcal{O}_{\mathcal{F}}(G), \mathcal{C}). \]
For $F: \mathcal{S} \xrightarrow{} \mathcal{C}$, when the context is clear we will also refer to the resulting functor $\mathcal{O}_{\mathcal{F}}(G) \xrightarrow{} \mathcal{C}$ as $F$. It sends the $G$-set $G/H$ to $F((G/H)_{hG})\simeq F(BH)$. This construction is standard (see for example \cite{rezk2014global}). We obtain an assembly map over the subgroups in $\mathcal{F}$
 \[ \underset{\mathcal{O}_{\mathcal{F}}(G)}{\colim} \ F(BH) \xrightarrow{} F(BG).\]
When this is an equivalence, we can hope to reconstruct the value of $F$ on $BG$ from its values on $BH$ as $H$ varies over the subgroups in $\mathcal{F}$. In particular, in this case there is a homotopy colimit spectral sequence (when $\mathcal{C}=\text{Sp}$, for example):
\[ E^2_{p,q}=\underset{\mathcal{O}_{\mathcal{F}}(G)}{\text{colim}}^p \ \pi_q(F(BH)) \Rightarrow \pi_{p+q}(F(BG)).\]

We now prove a general assembly theorem for $L_TBG$ (recall that $L_TBG$ is the full subgroupoid of $LBG$ indexed by conjugacy classes lying in $T \subseteq G$).

\begin{prop}[Assembly for the free loop space]\label{induct-theta}
    Let $\mathcal{F}$ be a family of subgroups and set $T = \bigcup_{H \in \mathcal{F}} H \subseteq G$. Then the assembly map for the functor $L(-): \mathcal{S} \xrightarrow{} \mathcal{S}^\mathrm{Fr}$ gives an equivalence of spaces with Frobenius lifts onto $L_TBG$:
    \[ \underset{\mathcal{O}_{\mathcal{F}}(G)}{\colim} \ LBH \xrightarrow{\simeq} L_TBG \subseteq LBG. \]
\end{prop}
\begin{proof}
    Colimits and equivalences in the presheaf category $\mathcal{S}^\mathrm{Fr}$ are computed underlying, so it is enough to show that the assembly map is an equivalence of spaces. The colimit can be described as the geometric realisation of the Grothendieck construction $\int_{\mathcal{O}_\mathcal{F}(G)}LBH$ (\cite[Theorem 1.2]{thomason1979homotopy}), which is a $1$-category with an explicit description in terms of objects and morphisms.

    The objects in $\int_{\mathcal{O}_\mathcal{F}(G)}LBH$ are given by triples $(X,x,g)$, where $X \in \mathcal{O}_\mathcal{F}(G)$ is a transitive $G$-set with isotropy in $\mathcal{F}$, $x\in X$ is a point in $X$, $g \in G$ is such that $g\cdot x=x$. A map from $(X, x, g)$ to $(Y,y,h)$ is a pair $(\phi, \alpha)$, where $\phi:X \xrightarrow{} Y$ is a map of $G$-sets, and $\alpha \in G$ is such that $\alpha \cdot \phi(x)=y$ and $h\alpha=\alpha g$. 

    Let $F$ denote the functor $\int_{\mathcal{O}_\mathcal{F}(G)}LBH \xrightarrow{} L_TBG$ induced by $LBH \xrightarrow{} L_TBH \subseteq LBG$. By Quillen's theorem A (\cite{quillen2006higher}), it is enough to see that for any object $g \in LBG$, the slice category $F/g$ is contractible. But from the above description of $\int_{\mathcal{O}_\mathcal{F}(G)}LBH$, the slice category is equivalent to the following poset:
    \[ F/g \simeq \{H \in \mathcal{F} \ | \ g \in H \} \]
    But by assumption, the poset is non-empty, and it has a minimal element, since it is closed under intersections. Hence the geometric realisation of the poset is contractible, as required, and the assembly map onto $L_TBG$ is an equivalence.
\end{proof}

We record two special cases. The first is the most important example of assembly for the free loop space: it says that $LBG$ may be recovered as the colimit over cyclic subgroups. Later we will deduce corresponding cyclic assembly results for invariants such as $\thh$ from the result for $LBG$.

\begin{cor}
    Let $\mathcal{F}_\mathrm{cyc}$ be the collection of cyclic subgroups. The assembly map for $L(-)$ over $\mathcal{F}_\mathrm{cyc}$ is an equivalence of spaces with Frobenius lifts:
    \[ \underset{\cyclicorbit(G)}{\colim} \ LBH \xrightarrow{\simeq} LBG. \]
\end{cor}

The second occurs where we take $\mathcal{F}$ to be the family of subgroups of a normal subgroup $N$.

\begin{cor}\label{cor-LBN}
    Suppose $N$ is a normal subgroup of $G$. Then there is an equivalence of spaces with Frobenius lift:
    \[LBN_{hG/N} \xrightarrow{\simeq} L_NBG\]

\end{cor}
\begin{proof}
    Take $\mathcal{F}=\{H \ | \ H \leq N\}$. By Proposition \ref{induct-theta}, we have that 
    \[ \underset{\mathcal{O}_{\mathcal{F}}(G)}{\colim} \ LBH \xrightarrow{\simeq} L_NBG. \]
    But the full subcategory on $G/N$ is cofinal in $\mathcal{O}_{\mathcal{F}}(G)$, so we have an equivalence
    \[ \underset{\mathcal{O}_{\mathcal{F}}(G)}{\colim} \ LBH \simeq LBN_{h{G/N}} \]
    as required.
\end{proof}

\section{$\thh$ and $\tc$ of group rings}
We now use our analysis of $LBG$ to deduce results on $\thh$ and $\tc$ of group rings. We begin with some background on these invariants, and how they can be described in the case of group rings. We then describe the filtrations on $\thh$ and $\tc$ arising from the filtration on $LBG$, and analyse the resulting spectral sequence for $\tc(R[G])$. Finally, we use Proposition \ref{induct-theta} to prove assembly theorems for $\thh$ and $\tc$. 

\subsection{Cyclotomic structures and Frobenius lifts}
For a ring $R$, the topological Hochschild homology $\thh(R)$ is the Hochschild homology relative to the sphere spectrum. Importantly, $\thh(R)$ has the structure of a cyclotomic spectrum.  Recall that a cyclotomic spectrum (in the sense of \cite{nikolaus2018topological}) is a spectrum $X$ with $S^1$-action, along with $S^1$-equivariant maps $X \xrightarrow{} X^{tC_p}$ for each prime $p$, where $(-)^{tC_p}$ denotes the $C_p$-Tate construction. We then define the topological cyclic homology $\tc(R)$ as the mapping spectrum in the category of cyclotomic spectra from the unit to $\thh(R)$. For a full account of $\thh$, $\tc$ and cyclotomic spectra, see \cite{nikolaus2018topological}. In the case where $R$ is a group ring, the cyclotomic structure admits a particularly convenient formulation in terms of the free loop space $LBG$, as we now describe.

The category of spectra with Frobenius lifts is deeply connected to the category of cyclotomic spectra. This is what motivates the study of objects with Frobenius lifts. If $X$ is a spectrum with Frobenius lifts, then we can (in particular) produce equivariant maps $X \xrightarrow{} X^{hC_p}$, using the action of $p \in \mathbb{N}^\times \subseteq \mathbb{W}$. Postcomposing the Frobenius lifts $X \xrightarrow{} X^{hC_p}$ with the canonical maps $X^{hC_p} \xrightarrow{} X^{tC_p}$ equips $X$ with a cyclotomic structure. We will also refer to the resulting cyclotomic spectrum as $X$ (although note that spectra with Frobenius lifts are not a subcategory of cyclotomic spectra).

\begin{definition}
    Let the canonical functor $\Sp^\mathrm{Fr} \xrightarrow{} \CycSp$ be the functor that gives $X$ a cyclotomic structure via the maps $X \xrightarrow{} X^{hC_p} \xrightarrow{} X^{tC_p}$ for each $p$.

    Similarly, let the canonical functor $\Sp^{p\mathrm{Fr}} \xrightarrow{} p\CycSp$ be the functor that gives $X$ a $p$-cyclotomic structure via the map $X \xrightarrow{} X^{hC_p} \xrightarrow{} X^{tC_p}$.

    For a formal construction, see \cite[Construction 2.3.7]{mccandless2021curves}.
\end{definition}

Now we may consider the composite
\[ \mathcal{S} \xrightarrow{L(-)} \mathcal{S}^\mathrm{Fr} \xrightarrow{} \Sp^\mathrm{Fr} \xrightarrow{} \CycSp. \]
This assigns to a space $X$ the spectrum $\mathbb{S}[LX]$, so we observe that $\mathbb{S}[LX]$ carries a canonical cyclotomic structure. Importantly, the cyclotomic spectrum $\mathbb{S}[LX]$ has an alternative description in terms of $\thh$. The based loop space $\Omega X$ is an $\mathbb{E}_1$-group, and the suspension spectrum (or group ring) $\mathbb{S}[\Omega X]$ is thus an algebra in spectra. Taking $\thh$ of this algebra recovers the cyclotomic spectrum $\mathbb{S}[LX]$.

\begin{theorem}[$\thh$ of group rings]\label{thh-group-rings}
    For $X \in \mathcal{S}$ connected, there is an equivalence of cyclotomic spectra
    \[ \thh(\mathbb{S}[\Omega X]) \simeq  \mathbb{S}[LX]. \]
\end{theorem}
\begin{proof}
    See \cite[Corollary IV.3.3]{nikolaus2018topological}. Note that they only consider the $p$-local case, but the global case follows immediately from their proof.
\end{proof}

The above theorem lets us compute $\thh$ of spherical group rings via the free loop space. Similarly, if $R\in \mathrm{CAlg}(\Sp)$ is any ring of coefficients, by symmetric monoidality of $\thh$ we obtain an equivalence of $\thh(R)$-modules in cyclotomic spectra
\[ \thh(R[\Omega X]) \simeq \thh(R) \otimes \thh(\mathbb{S}[\Omega X]) \simeq \thh(R) \otimes \mathbb{S}[LX]. \]
Hence we can also understand the cyclotomic spectrum $\thh(R[\Omega X])$ via the free loop space. In the next part of this section, we analyse the space with Frobenius lifts $LBG$ (for a discrete group $G$), and use this to produce decompositions and filtrations of $\thh(R[G])$.

\begin{remark} 
    We can immediately deduce the result for $\thh$ relative to $R$. The equivalence 
    \[  \thh(R[\Omega X]) \simeq \thh(R) \otimes \mathbb{S}[LX] \]
    is (in particular) an equivalence of $\thh(R)$-modules in $\Sp$, so we can compute
    \[ \thh(R[\Omega X]/R) \simeq \thh(R[\Omega X])\otimes_{\thh(R)} R \simeq \mathbb{S}[LX] \otimes R \simeq R[LX]. \]
    The right hand side, $R[LX]$, is an $R$-module with Frobenius lifts. Relative $\thh$ does not have a cyclotomic structure, so this cannot be promoted to an equivalence of cyclotomic structures. It does, however, tell us that Hochschild homology of group algebras always carries the structure of Frobenius lifts (by transport of structure).
\end{remark}

\begin{example}
    Let $G=S_3$ and $p=3$. By Example \ref{ex-lbg-decomp}, there is a decomposition 
     \[ LBS_3 \simeq L_{[1]\cup[(123)]}BS_3 \sqcup  L_{[(12)]}BS_3\]
    as spaces with a $p$-Frobenius lift. Hence for $R \in \mathrm{CAlg}(\Sp)$ we have a decomposition
    \[ \thh(R[S_3]) \simeq \thh(R) \otimes \mathbb{S}[LBS_3]\simeq \thh(R)\otimes\mathbb{S}[L_{[1]\cup[(123)]}BS_3] \oplus \thh(R) \otimes \mathbb{S}[L_{[(12)]}BS_3] \]
    of $p$-cyclotomic spectra. Finally, applying $\tc(-;p)$, we have a decomposition
    \[ \tc(R[S_3];p) \simeq \tc\big(\thh(R)\otimes\mathbb{S}[L_{[1]\cup[(123)]}BS_3];p \big) \oplus \tc\big( \thh(R) \otimes \mathbb{S}[L_{[(12)]}BS_3];p \big). \]
\end{example}

\subsection{Assembly for $\thh$ and $\tc$}

We now turn to assembly theorems for $\thh$ and $\tc$. This will be deduced directly from the assembly results for $LBG$ (Proposition \ref{induct-theta}) and the equivalence $\thh(R[G]) \simeq \thh(R) \otimes \mathbb{S}[LBG]$.

The suspension functor $\mathcal{S}^\mathrm{Fr} \xrightarrow{} \Sp^\mathrm{Fr}$ gives rise to a functor $\mathcal{O}(G) \xrightarrow{} \Sp^\mathrm{Fr}$ sending $G/H$ to $\mathbb{S}[LBH]$. Postcomposing with the canonical functor and tensoring with $\thh(R)$ gives a functor $\mathcal{O}(G) \xrightarrow{} \CycSp$ sending $G/H$ to $\thh(R) \otimes \mathbb{S}[LBH]$. Observe that while $\thh(R) \otimes \mathbb{S}[LBH]$ is equivalent to $\thh(R[H])$, the functoriality over $\mathcal{O}(G)$ cannot be described on the level of rings.

\begin{prop}[Assembly for $\thh$]\label{induct-thh}
    Let $R \in \mathrm{CAlg}(\Sp)$ be a ring, and $\mathcal{F}$ a family of subgroups. Set $T = \bigcup_{H \in \mathcal{F}} H \subseteq G$. Then the assembly map for the functor
    \[ \thh(R) \otimes \mathbb{S}[L(-)]: \mathcal{S} \xrightarrow{} \CycSp \]
    gives an equivalence
    \[ \underset{\mathcal{O}_{\mathcal{F}}(G)}{\colim} \ \thh(R) \otimes \mathbb{S}[ LBH] \xrightarrow{\simeq} \thh(R) \otimes \mathbb{S}[ L_TBG] \]
    of cyclotomic spectra.
\end{prop}
\begin{proof}
    By Proposition \ref{induct-theta}, the map of spaces with Frobenius lifts
    \[ \underset{\mathcal{O}_{\mathcal{F}}(G)}{\colim} \ LBH \xrightarrow{} L_TBG. \]
    is an equivalence.
    
    The suspension functor $\mathbb{S}[-]: \mathcal{S}^\mathrm{Fr} \xrightarrow{} \Sp^\mathrm{Fr}$, the canonical functor $\Sp^\mathrm{Fr} \xrightarrow{} \CycSp$ and tensoring with $\thh(R)$ preserve colimits, so 
    \[ \underset{\mathcal{O}_{\mathcal{F}}(G)}{\colim} \ \thh(R) \otimes \mathbb{S}[ LBH] \simeq \thh(R) \otimes \mathbb{S}[\underset{\mathcal{O}_{\mathcal{F}}(G)}{\colim} \ LBH] \]
    as cyclotomic spectra.
\end{proof}
\begin{remark}
    The cyclotomic spectrum $\thh(R)$ is a commutative algebra in cyclotomic spectra, and the equivalence of Proposition \ref{induct-thh} is moreover an equivalence of $\thh(R)$-modules in cyclotomic spectra.
\end{remark}

After applying the equivalence $\thh(R) \otimes \mathbb{S}[ LBH] \simeq \thh(R[H])$, Proposition \ref{induct-thh} becomes 
\[ \underset{\mathcal{O}_{\mathcal{F}}(G)}{\colim} \ \thh(R[H]) \xrightarrow{\simeq} \thh(R) \otimes \mathbb{S}[ L_TBG] \]
and in particular if $T =G$,
\[ \underset{\mathcal{O}_{\mathcal{F}}(G)}{\colim} \ \thh(R[H]) \xrightarrow{\simeq} \thh(R[G]). \]
While the latter formulation is more intuitive, it obscures what exactly the functoriality in $\mathcal{O}_{\mathcal{F}}(G)$ is.

We may now show the assembly theorem for $\tc$, using that $\tc(-;p)$ commutes with colimits of connective spectra up to $p$-completion.

\begin{theorem}[Assembly for $\tc$]\label{induct-tc}
    Let $R \in \mathrm{CAlg}(\Sp)$ be connective, and $\mathcal{F}$ a family of subgroups. Set $T = \bigcup_{H \in \mathcal{F}} H \subseteq G$. Then the assembly map for $\tc$ gives an equivalence after $p$-completion:
    \[ \underset{\mathcal{O}_{\mathcal{F}}(G)}{\colim} \ \tc \big(  \thh(R) \otimes \mathbb{S}[LBH];p\big) \xrightarrow{\simeq}  \tc \big(  \thh(R) \otimes \mathbb{S}[L_TBG];p\big). \]
    In particular, if $T=G$, then the assembly map is an equivalence after $p$-completion:
     \[ \underset{\mathcal{O}_{\mathcal{F}}(G)}{\colim} \ \tc \big(  \thh(R) \otimes \mathbb{S}[LBH];p\big) \xrightarrow{\simeq}  \tc \big(  \thh(R) \otimes \mathbb{S}[LBG];p\big) \simeq \tc(R[G];p). \]
\end{theorem}
\begin{proof}
By \cite[Theorem 2.7]{clausen2021}, the functor $\tc(-)/p$ commutes with colimits of connective cyclotomic spectra, so from Proposition \ref{induct-thh}, we see that
 \[ \underset{\mathcal{O}_{\mathcal{F}}(G)}{\colim} \ \tc \big(  \thh(R) \otimes \mathbb{S}[LBH]\big)/p \xrightarrow{\simeq}  \tc \big(  \thh(R) \otimes \mathbb{S}[L_TBG]\big)/p \]
 and hence the assembly map for $\tc(-)$ is an equivalence after $p$-completion. We conclude by the standard equivalence $\tc(X)^\wedge_p \simeq \tc(X;p)^\wedge_p$.
\end{proof}
\begin{remark}\label{p-comp-remark}
    Suppose $R$ is an $\mathbb{F}_p$-algebra. Then $\thh(R) \otimes \mathbb{S}[L_TBG]$ is an $\mathbb{F}_p$-module, so $p$-complete, which implies that $\tc \big(  \thh(R) \otimes \mathbb{S}[L_TBG];p\big)$ is already $p$-complete. Indeed, there is a cofibre sequence
    \[ \tc \big(  \thh(R) \otimes \mathbb{S}[L_TBG];p\big) \xrightarrow{} (\thh(R) \otimes \mathbb{S}[L_TBG])^{hS^1} \xrightarrow{} \big((\thh(R) \otimes \mathbb{S}[L_TBG])^{tC_p}\big)^{hS^1} \]
    in which the middle and right terms are $p$-complete (via \cite[Lemma I.2.9]{nikolaus2018topological} and closure of $p$-complete objects under limits). Moreover, in this case $\tc \big(  \thh(R) \otimes \mathbb{S}[L_TBG];p\big)$ is a module over $\tc(\mathbb{F}_p)_{\geq 0} \simeq \mathbb{Z}_p$.
\end{remark}

The most important application of Theorem \ref{induct-tc} is that $\tc$ has assembly for the family of cyclic subgroups. This mostly appears as \cite[Theorem 1.2]{luck2019assembly}. Note, however, that we do not require the family of cyclic subgroups to have a classifying space of finite type, at the expense of taking $p$-completion.

\begin{cor}[\cite{luck2019assembly}]\label{cyc-induct-tc}
    The cyclic assembly map for $\tc$ is an equivalence for $R$ connective after $p$-completion:
     \[ \underset{\mathcal{O}_{\mathcal{F}_\mathrm{cyc}}(G)}{\colim} \ \tc \big(  \thh(R) \otimes \mathbb{S}[LBH];p\big)\xrightarrow{\simeq}  \tc \big(  \thh(R) \otimes \mathbb{S}[LBG];p\big) \simeq \tc(R[G];p). \]
\end{cor}

Corollary \ref{cyc-induct-tc} suggests that we can calculate topological cyclic homology of finite group rings from the topological cyclic homology of finite cyclic group rings. One would ideally like to know the value of $\tc_*(R[C_{p^n}])$. For groups of order $p$ and arbitrary coefficients, this is calculated up to extension in \cite[Theorem 1.4.1]{hesselholt2020topological}. In characteristic $p$, the situation is more straightforward. In this case there is an isomorphism $R[C_{p^n}]\cong R[x]/(x^{p^n})$, and the topological cyclic homology of truncated polynomial algebras has been evaluated by Hesselholt and Madsen (\cite{hesselholt1997cyclic}, \cite{speirs2020k}). In the next section, we will use this to calculate $\tc(k[C_p^r])$ for a perfect field $k$.

A further class of examples arises from taking $\mathcal{F}$ to be the family of subgroups of a normal subgroup $N$. For such a family, we can recover the topological cyclic homology associated to the subspace $L_NBG$.

\begin{prop}\label{tc-normal}
    Let $N \lhd G$ be a normal subgroup. Then for $R \in \mathrm{CAlg}(\mathrm{Sp})$ connective, after $p$-completion:
    \[ \tc(R[N];p)_{hG/N} \xrightarrow{\simeq} \tc \big(\thh(R) \otimes \mathbb{S}[L_NBG] ; p \big). \]
\end{prop}
\begin{proof}
    Consider the family $\mathcal{F} = \{ H \ | \ H \leq N \}$. By Theorem \ref{induct-tc}, we have that 
    \[ \underset{\mathcal{O}_{\mathcal{F}}(G)}{\colim} \ \tc \big(  \thh(R) \otimes \mathbb{S}[LBH];p\big) \xrightarrow{\simeq}  \tc \big(  \thh(R) \otimes \mathbb{S}[L_NBG];p\big). \]
    But, as in the proof of Corollary \ref{cor-LBN}, we have an equivalence
    \[ \underset{\mathcal{O}_{\mathcal{F}}(G)}{\colim} \ \tc \big(  \thh(R) \otimes \mathbb{S}[LBH];p\big) \simeq \tc(R[N];p)_{hG/N}. \qedhere\]
\end{proof}
\begin{remark}
    Again, one must be careful with the functoriality in this statement. The $G/N$ action on $\tc(R[N];p)$ arises from the $G/N$ action on the space $BN$, not on the level of rings. In particular, the action may be non-trivial even for central extensions.
\end{remark}

\subsection{The family of $p'$-subgroups}
We now examine the assembly map when $\mathcal{F}$ is a family consisting of subgroups with order prime to $p$ (henceforth referred to as $p'$-subgroups). In the modular representation theory of finite groups, the character theory is controlled by $p'$-subgroups. Indeed, the Brauer character provides a rational equivalence between $K_0(k[G])$ and the group of $W(k)$-class functions on conjugacy classes with order prime to $p$. An analogous statement cannot hold for $\tc$: indeed, for $k$ a field of characteristic $p$ and $P$ a $p$-group, the spectrum $\tc(k[P])$ will in general differ from $\tc(k)$ in positive degrees. Nevertheless, we now show that an analogous statement does hold in negative degrees, and up to a surjection in degree $0$.

We begin by constructing a filtration on the space $LBG$, compatible with the $p$-Frobenius lifts.

\begin{definition}
    Let $\Lambda_i \subseteq G$ be the set of elements of $G$ with order dividing $p^jm$ for $j\leq i$, $p\nmid m$. We will write $L_iBG$ for $L_{\Lambda_i}BG$. 
\end{definition}

\begin{prop}\label{LBG-filtration}
    Suppose $|G| = p^nm$. Then
    \[ L_0BG \subseteq L_1BG \subseteq \ldots \subseteq L_nBG  \]
    is an $n$-step filtration of $LBG$ as a space with a $p$-Frobenius lift.

    Moreover, the action of $p \in \mathbb{N}^\times$ on $L_iBG$ for $i>0$ factors as
    % https://q.uiver.app/#q=WzAsMyxbMCwxLCJMX3tpLTF9QkciXSxbMSwxLCJMX2lCRyJdLFsxLDAsIkxfaUJHIl0sWzAsMSwiIiwwLHsic3R5bGUiOnsidGFpbCI6eyJuYW1lIjoiaG9vayIsInNpZGUiOiJ0b3AifX19XSxbMiwxLCJcXGNkb3QgcCJdLFsyLDAsIiIsMSx7InN0eWxlIjp7ImJvZHkiOnsibmFtZSI6ImRhc2hlZCJ9fX1dXQ==
\[\begin{tikzcd}
	& {L_iBG} \\
	{L_{i-1}BG} & {L_iBG}
	\arrow[dashed, from=1-2, to=2-1]
	\arrow["{\cdot p}", from=1-2, to=2-2]
	\arrow[hook, from=2-1, to=2-2]
\end{tikzcd}\]
\end{prop}
\begin{proof}
    By definition, the subsets $\Lambda_i$ are closed under conjugation and $g \mapsto g^p$, so the first part follows from Proposition \ref{sub-frob-lift}.

    For the second, note again that $L_{i-1}BG \subseteq L_{i}BG$ is an inclusion of components, so it suffices to see that the factorisation exists on $\pi_0$. On $\pi_0(LBG)$, $p$ acts by sending a conjugacy class $[g]$ to $[g^p]$, so sends the conjugacy classes in $\Lambda_i$ to those in $\Lambda_{i-1}$ for $i>0$.
\end{proof}

The filtration of Proposition \ref{LBG-filtration} gives a filtration on $\thh(R[G])$ as a $p$-cyclotomic spectrum upon taking suspension spectra and tensoring with $\thh(R)$.

We can identify the associated graded explicitly. Recall that if $X \xrightarrow{} Y$ is a map of $p$-cyclotomic spectra, then the cofibre gains the structure of a $p$-cyclotomic spectrum via:

% https://q.uiver.app/#q=WzAsNixbMCwwLCJYIl0sWzAsMSwiWF57dENfcH0iXSxbMSwwLCJZIl0sWzEsMSwiWV57dENfcH0iXSxbMiwwLCJaIl0sWzIsMSwiWl57dENfcH0iXSxbMCwyXSxbMSwzXSxbMCwxXSxbMiwzXSxbNCw1LCIiLDEseyJzdHlsZSI6eyJib2R5Ijp7Im5hbWUiOiJkYXNoZWQifX19XSxbMiw0XSxbMyw1XV0=
\[\begin{tikzcd}
	X & Y & Z \\
	{X^{tC_p}} & {Y^{tC_p}} & {Z^{tC_p}}
	\arrow[from=1-1, to=1-2]
	\arrow[from=1-1, to=2-1]
	\arrow[from=1-2, to=1-3]
	\arrow[from=1-2, to=2-2]
	\arrow[dashed, from=1-3, to=2-3]
	\arrow[from=2-1, to=2-2]
	\arrow[from=2-2, to=2-3]
\end{tikzcd}\]

Moreover, $Z$ (along with the map $Z \xrightarrow{} Z^{tC_p}$) is the cofibre in $p$-cyclotomic spectra.

\begin{lemma}\label{filtration-thh-rg}
    Suppose that $|G|=p^nm$ ($p\nmid m$), and $R \in \mathrm{CAlg}(\mathrm{Sp})$. There is a length $n$ filtration of $\thh(R[G])$ as a $p$-cyclotomic spectrum
    \[ \thh(R) \otimes \mathbb{S}[L_0BG] \xrightarrow{} \thh(R) \otimes \mathbb{S}[L_1BG] \xrightarrow{} \ldots \xrightarrow{ } \thh(R) \otimes \mathbb{S}[L_nBG] \simeq \thh(R[G])\]
    and for $i>0$, the associated graded is
    \[ \thh(R) \otimes  \mathbb{S}[L_{\Lambda_i - \Lambda_{i-1}}BG] \simeq \bigoplus_{[g], \ g \in \Lambda_i - \Lambda_{i-1}} \thh(R) \otimes \mathbb{S}[BC_G(g)] \]
    with $p$-cyclotomic Frobenius $0$.
\end{lemma}
\begin{proof}
    The existence of the filtration follows from Proposition \ref{LBG-filtration}, by applying the functor $\thh(R) \otimes \mathbb{S}[-]$. Let $C_i$ denote $\mathrm{cofib}(\mathbb{S}[L_{i-1}BG] \xrightarrow{} \mathbb{S}[L_iBG])$ (for $i>0$). As a spectrum with $S^1$-action, $C_i$ is $\mathbb{S}[L_{\Lambda_i-\Lambda_{i-1}}BG]$, so that by the centraliser decomposition
    \[ C_i \simeq \bigoplus_{[g], \ g \in \Lambda_i - \Lambda_{i-1}} \mathbb{S}[BC_G(g)]. \]

    The cyclotomic Frobenius on $C_i$ is induced by the Frobenius lift $C_i \xrightarrow{} C_i^{hC_p}$. To see that the latter is $0$, note that by Proposition \ref{LBG-filtration}, the Frobenius lift on $L_{i}BG$ factors through $L_{i-1}BG^{hC_p}$, so we obtain a factorisation of the $p$-Frobenius lift on $\mathbb{S}[L_{i}BG]$, and hence the induced map on the cofibre is $0$:

    % https://q.uiver.app/#q=WzAsOCxbMCwwLCJcXG1hdGhiYntTfVtMX3tpLTF9QkddIl0sWzEsMCwiXFxtYXRoYmJ7U31bTF9pQkddIl0sWzIsMCwiQ19pIl0sWzAsMSwiXFxtYXRoYmJ7U31bTF97aS0xfUJHXntoQ19wfV0iXSxbMSwxLCJcXG1hdGhiYntTfVtMX2lCR157aENfcH1dIl0sWzAsMiwiXFxtYXRoYmJ7U31bTF97aS0xfUJHXV57aENfcH0iXSxbMSwyLCJcXG1hdGhiYntTfVtMX3tpLTF9QkddXntoQ19wfSJdLFsyLDIsIkNfaV57aENfcH0iXSxbMCwzXSxbMyw1XSxbMCwxXSxbMSw0XSxbNCw2XSxbMyw0XSxbNSw2XSxbMSwyXSxbNiw3XSxbMiw3LCIwIl0sWzEsMywiIiwxLHsic3R5bGUiOnsiYm9keSI6eyJuYW1lIjoiZGFzaGVkIn19fV1d
\[\begin{tikzcd}
	{\mathbb{S}[L_{i-1}BG]} & {\mathbb{S}[L_iBG]} & {C_i} \\
	{\mathbb{S}[L_{i-1}BG^{hC_p}]} & {\mathbb{S}[L_iBG^{hC_p}]} \\
	{\mathbb{S}[L_{i-1}BG]^{hC_p}} & {\mathbb{S}[L_iBG]^{hC_p}} & {C_i^{hC_p}}
	\arrow[from=1-1, to=1-2]
	\arrow[from=1-1, to=2-1]
	\arrow[from=1-2, to=1-3]
	\arrow[dashed, from=1-2, to=2-1]
	\arrow[from=1-2, to=2-2]
	\arrow["0", from=1-3, to=3-3]
	\arrow[from=2-1, to=2-2]
	\arrow[from=2-1, to=3-1]
	\arrow[from=2-2, to=3-2]
	\arrow[from=3-1, to=3-2]
	\arrow[from=3-2, to=3-3]
\end{tikzcd}\]

    Now the results for $\thh(R[G])\simeq \thh(R) \otimes \mathbb{S}[LBG]$ follow from tensoring with $\thh(R)$. Note that for $X,Y \in p\CycSp$, the tensor product is the $S^1$-spectrum $X \otimes Y$ with cyclotomic Frobenius 
    \[ X \otimes Y \xrightarrow{} X^{tC_p} \otimes Y^{tC_p} \xrightarrow{}  (X\otimes Y)^{tC_p}\]
    so that if the cyclotomic Frobenius on $Y$ is $0$, then the cyclotomic Frobenius on $X\otimes Y$ is $0$.
\end{proof}

The only associated graded we have not described is that in degree $0$. As a spectrum with $S^1$-action, this is $ \thh(R) \otimes \mathbb{S}[L_0BG]$, so it is the summand of $\thh(R[G])$ indexed by the $p'$-conjugacy classes. If $k$ is a large enough field of characteristic $p$, then $K_0(k[G])$ is equivalent to $\tc_0(k[G])$ (up to $p$-completion), and is free of rank the number of $p'$-conjugacy classes. Hence we can view the higher terms in this filtration as the contributions to the higher $K$-theory from the $p$-singular conjugacy classes, which do not contribute to $K_0$. We can make this precise via the spectral sequence for $\tc(R[G])$ arising from the filtration on $\thh(R[G])$.

\begin{prop}[Spectral sequence for TC]\label{tc-ss}
    Suppose that $|G|=p^nm$ ($p \nmid m$), and $R \in \mathrm{CAlg}(\Sp)$ is bounded below. Let $C$ be the cofibre of the map induced by the inclusion of the $p'$-elements in $G$,
    \[ \thh(R) \otimes \mathbb{S}[L_0BG] \xrightarrow{} \thh(R) \otimes \mathbb{S}[LBG]\simeq \thh(R[G]) \xrightarrow{} C. \]

    There is a spectral sequence converging to $\tc(C;p)^\wedge_p$:
    \begin{align*}
        E^1_{kl} &\simeq \pi_{k+l}\Big( \ \Sigma\big(\mathbb{S}[L_{\Lambda_k - \Lambda_{k-1}}BG] \otimes \thh(R)\big)_{hS^1} \Big)^\wedge_p \\
        &\simeq \pi_{k+l-1}\Big( \ \big(\bigoplus_{[g], \  g \in \Lambda_k - \Lambda_{k-1}} R[BC_G(g)] \otimes_R \thh(R)\big)_{hS^1} \Big)^\wedge_p\\
        &\Rightarrow \tc(C;p)^\wedge_p
    \end{align*}
\end{prop}
\begin{proof}
    By Lemma \ref{filtration-thh-rg}, there is an $n$-step filtration on $C$ with graded parts \[ \thh(R) \otimes  \mathbb{S}[L_{\Lambda_i - \Lambda_{i-1}}BG] \simeq \bigoplus_{[g], \ g \in \Lambda_i - \Lambda_{i-1}} \thh(R) \otimes \mathbb{S}[BC_G(g)] \]
    for $1 \leq i \leq n$. Applying $\tc(-;p)$, we obtain a filtration on $\tc(C;p)^\wedge_p$ with associated graded
    \[\tc\big(\thh(R) \otimes  \mathbb{S}[L_{\Lambda_i - \Lambda_{i-1}}BG];p\big).\]

    For a bounded below $p$-cyclotomic spectrum $X$, there is a fibre sequence (\cite[Lemma II.4.2.]{nikolaus2018topological})
    \[ \tc(X;p) \xrightarrow{} X^{hS^1}\xrightarrow{\varphi_p^{hS^1}-\mathrm{can}} (X^{tC_p})^{hS^1} \]
    so that if the cyclotomic structure map $\varphi_p$ is $0$, there is an equivalence $\tc(X;p) \simeq (\Sigma 
     \ X_{hS^1})^\wedge_p$ (\cite[Corollary I.4.3]{nikolaus2018topological}). Now the spectrum 
     \[ \Sigma\big(\mathbb{S}[L_{\Lambda_k - \Lambda_{k-1}}BG] \otimes \thh(R)\big)_{hS^1} \]
    is bounded below, so the homotopy groups of the $p$-completion are the $p$-completion of the homotopy groups. This completes the identification of the $E^1$ page.
\end{proof}

In particular, we see that for $R$ connective, the spectrum $\tc(C;p)^\wedge_p$ for $C$ as in Proposition \ref{tc-ss} is $1$-connective.

\begin{cor}\label{tc-leq0-cor}
    Suppose that $R$ is connective. Then the map $L_0BG \xrightarrow{} LBG$ induces a map
    \[ \tc\big(\thh(R) \otimes \mathbb{S}[L_0BG];p\big)^\wedge_p \xrightarrow{} \tc\big(\thh(R) \otimes \mathbb{S}[LBG];p\big)^\wedge_p \simeq \tc(R[G];p)^\wedge_p \]
    with $1$-connective cofibre. In particular, it induces an isomorphism on $\pi_j$ for $j<0$, and a surjection on $\pi_0$.
\end{cor}
\begin{proof}
    Let $C$ be as in the statement of Proposition \ref{tc-ss}. The cofibre of the map in question is $\tc(C;p)^\wedge_p$, so it suffices to see that $\tc(C;p)^\wedge_p$ is $1$-connective. But for $R$ connective, the spectrum
    \[ \big(\bigoplus_{[g], \  g \in \Lambda_k - \Lambda_{k-1}} R[BC_G(g)] \otimes_R \thh(R)\big)_{hS^1} \]
    is again connective, so the spectral sequence is concentrated in degrees $k+l>0$. Hence $\pi_j\big(\tc(C;p)^\wedge_p\big)=0$ for $j \leq 0$.
\end{proof}

We now use Theorem \ref{induct-tc} to describe the spectra $\tc\big(\thh(R) \otimes \mathbb{S}[L_rBG];p\big)$ via assembly.

\begin{lemma}\label{induct-cyc-pr}
    For $r \geq 0$, let $\mathcal{F}_\mathrm{cyc}^r$ be the family of cyclic subgroups of order $p^sm$ for $s \leq r$, $p \nmid m$. Then the assembly map gives an equivalence after $p$-completion (for $R$ connective)
    \[ \underset{\mathcal{O}_{\mathcal{F}_\mathrm{cyc}^r}(G)}{\colim} \ \tc \big(  \thh(R) \otimes \mathbb{S}[LBH] ;p \big) \xrightarrow{\simeq}  \tc \big(  \thh(R) \otimes \mathbb{S}[L_rBG] ;p\big). \]
\end{lemma}
\begin{proof}
    This is immediate from Theorem \ref{induct-tc}. 
\end{proof}

An immediate corollary of Lemma \ref{induct-cyc-pr} is that assembly from the family of cyclic subgroups of order prime to $p$ controls $\tc(R[G])$ in negative degrees, and up to a surjection in degree $0$.

\begin{cor}\label{tc-cyc-prime}
    Let $\mathcal{F}_\mathrm{cyc}^0$ be the family of cyclic subgroups of order prime to $p$, and $R$ connective. After $p$-completion, the assembly map
    \[ \underset{\mathcal{O}_{\mathcal{F}_\mathrm{cyc}^0}(G)}{\colim} \ \tc \big(  \thh(R) \otimes \mathbb{S}[LBH] ;p \big) \simeq \underset{\mathcal{O}_{\mathcal{F}_\mathrm{cyc}^0}(G)}{\colim} \ \tc \big( R[H] ;p \big) \xrightarrow{} \tc(R[G];p) \]
    has $1$-connective cofibre. In particular, it induces an isomorphism on $\pi_j$ for $j<0$, and a surjection on $\pi_0$.
\end{cor}
\begin{proof}
    This is a consequence of Corollary \ref{tc-leq0-cor} and Lemma \ref{induct-cyc-pr}.
\end{proof}

\subsection{Frobenius groups}
Given a semi-direct product $G \cong K \rtimes H$, one would like to express $\tc(R[G])$ in terms of $\tc(R[K])$ and $\tc(R[H])$. This is not possible in general, but if the semi-direct product happens to be a Frobenius group then Proposition \ref{induct-theta} can be used to produce such a description, as we now show. Recall that $K \rtimes H$ is a Frobenius group if $H \cap H^g$ is trivial for all $g \in G \setminus H$. Normalisers of Sylow subgroups are often Frobenius groups, so such groups play an important role in $p$-local representation theory.

\begin{prop}\label{tc-frob-groups}
    Suppose $G\cong K \rtimes H$ is a Frobenius group. For $R \in \mathrm{CAlg}(\Sp)$ connective, the diagram

    % https://q.uiver.app/#q=WzAsNCxbMCwwLCJcXHRjKFIpX3toSH0iXSxbMiwwLCJcXHRjKFJbSF0pIl0sWzAsMiwiXFx0YyhSW0tdKV97aEh9Il0sWzIsMiwiXFx0YyhSW0ddKSJdLFswLDJdLFsyLDNdLFswLDFdLFsxLDNdXQ==
\[\begin{tikzcd}
	{\tc(R;p)_{hH}} && {\tc(R[H];p)} \\
	\\
	{\tc(R[K];p)_{hH}} && {\tc(R[G];p)}
	\arrow[from=1-1, to=1-3]
	\arrow[from=1-1, to=3-1]
	\arrow[from=1-3, to=3-3]
	\arrow[from=3-1, to=3-3]
\end{tikzcd}\]
    becomes a pushout after $p$-completion.
    If moreover the order of $H$ is prime to $p$, then the map $\tc(R[K];p)_{hH}\xrightarrow{} \tc(R[G];p)$ is split after $p$-completion.
\end{prop}
\begin{proof}
    By standard properties of Frobenius groups, we have that the conjugacy classes of $G$ may be partitioned into
    \begin{enumerate}
        \item conjugacy classes lying in $K$,
        \item and $[h]=\{kh' \ | \ k\in K,h' \in [h]_H\}$ for $1 \neq h \in H$.
    \end{enumerate}
    Furthermore, we have that $C_G(h)=C_H(h)$ for $1 \neq h\in H$, so that
    \[ LBG \simeq L_KBG \bigsqcup (LBH \setminus BH) \]
    as spaces. The above decomposition is compatible with the Frobenius lift structure if the order of $H$ is prime to $p$, by Proposition \ref{sub-frob-lift}. It follows that we have a pushout of cyclotomic spectra, with the horizontal maps split if the order of $H$ is prime to $p$:
    % https://q.uiver.app/#q=WzAsNCxbMCwwLCJcXG1hdGhiYntTfVtCSF0iXSxbMiwwLCJcXG1hdGhiYntTfVtMQkhdIl0sWzAsMiwiXFxtYXRoYmJ7U31bTF9LQkddIl0sWzIsMiwiXFxtYXRoYmJ7U31bTEJHXSJdLFswLDJdLFsyLDNdLFswLDFdLFsxLDNdXQ==
\[\begin{tikzcd}
	{\mathbb{S}[BH]} && {\mathbb{S}[LBH]} \\
	\\
	{\mathbb{S}[L_KBG]} && {\mathbb{S}[LBG]}
	\arrow[from=1-1, to=1-3]
	\arrow[from=1-1, to=3-1]
	\arrow[from=1-3, to=3-3]
	\arrow[from=3-1, to=3-3]
\end{tikzcd}\]
    Applying $\tc\big(\thh(R)\otimes -\big)$ to the upper horizontal map produces the assembly map $\tc(R)_{hH} \xrightarrow{} \tc(R[H])$, by Proposition \ref{tc-normal}. Similarly, applying $\tc\big(\thh(R)\otimes -\big)$ to the lower horizontal map produces the relative assembly map $\tc(R[K])_{hH} \xrightarrow{} \tc(R[G])$. Thus after applying $\tc$ we obtain the stated pushout.
\end{proof}
\begin{remark}
    For a Frobenius group $G \cong K \rtimes H$, the complex representation theory of $G$ is determined by that of $K$ and $H$. Indeed, the irreducible representations of $G$ are precisely those of $H$ (via inflation), along with the non-trivial irreducible representations of $K$ (via induction) up to $H$-conjugation. Proposition \ref{tc-frob-groups} is an exactly analogous statement.
\end{remark}

Proposition \ref{tc-frob-groups} is particularly useful when $R=k$ is a perfect algebra of characteristic $p$ and the order of $H$ is prime to $p$. In this case, $k[H]$ is semi-simple and so both $\tc(k)$ and $\tc(k[H])$ are $0$ in positive degrees (\cite[Corollary IV.4.10]{nikolaus2018topological}). Also $\tc(k[K])$ is a module over $\tc(\mathbb{F}_p)_{\geq 0} \simeq \mathbb{Z}_p$, so that taking $H$-orbits is exact by Maschke's theorem. Then Proposition \ref{tc-frob-groups} shows that
\[ \tc_i(k[G];p) \cong \tc_i(k[K];p)_{H} \]
for $i>0$.

\begin{example}
    Let $k$ be a perfect algebra of characteristic $p$, for $p>2$. Then
    \[ \tc_i(k[D_{2p}];p) \cong \tc_i(k[C_p];p)_{C_2} \]
    for $i>0$, where $C_2$ acts on $\tc_i(k[C_p])$ via the action on $C_p$.

    Suppose now $k$ has characteristic $2$ and contains a primitive $p^\mathrm{th}$ root of unity. Then $k[C_p]\simeq \prod_{j=1}^p k$, and $\tc(k[C_p];2)\simeq\tc(k)^{\oplus p}$, where the action of $C_2$ is free away from the image of $\tc(k;2) \xrightarrow{} \tc(k[C_p];2)$. The map $\tc(k)_{hC_2} \xrightarrow{} \tc(k[C_p];2)_{hC_2}$ is thus an equivalence in positive degrees and injective on $\pi_0$, and so by Proposition \ref{tc-frob-groups},
    \[\tc_i(k[D_{2p}];2) \cong  \tc_i(k[C_2];2)  \]
    for $i>0$.
\end{example}
\section{Elementary abelian group rings}
In this section, we compute the groups $\tc_i(k[C_p^r];p)$ for $k$ a finite field of characteristic $p$. Most of the intermediate results hold for any perfect field $k$ of characteristic $p$ (and we will need the case of an algebraically closed field in the proof of Proposition \ref{split-inj-ass}); we only restrict to finite fields in Theorem \ref{tc-el-al}. Our approach is to use Theorem \ref{induct-tc} to reduce to cyclic subgroups, in combination with the assembly map. The category $\cyclicorbit(C_p^r)$ is particularly simple. It contains $G/e$ and $G/H$ as $H$ varies over the $(p^r-1)/(p-1)$ cyclic subgroups of order $p$. There are no morphisms between distinct subgroups of order $p$.

We also calculate the action of $H \leq \mathrm{GL}_r(\mathbb{F}_p)$ on $\tc_i(k[C_p^r];p)$ for subgroups $H$ with the property that $H$ acts freely on $\mathbb{F}_p^r \setminus \{0\} $. A key example is the action of $C_{p^r-1}$ on $\tc_i(k[C_p^r])$ via the Singer cycle $C_{p^r-1} \cong \mathbb{F}_{p^r}^\times \leq  \mathrm{GL}_r(\mathbb{F}_p)$.

\subsection{Preliminaries}

We fix a perfect field $k$ of characteristic $p$. For ease of notation, we make the following conventions in this section.

\begin{definition}\label{tc-spc-def}
    Let $\tc(-;k)$ denote the functor $\mathcal{S} \xrightarrow{} \Sp$ given by
    \[ X \mapsto \tc\big( \thh(k) \otimes \mathbb{S}[LX];p\big)\]
    so that upon picking a basepoint we have an equivalence $\tc(X;k) \simeq \tc(k[\Omega X];p).$ Observe that $\tc(X;k)$ is $p$-complete, and that it is a module over $\tc(k)_{\geq0}\simeq \mathbb{Z}_p$ (see Remark \ref{p-comp-remark}).
\end{definition}

\begin{definition}
    For $\mathcal{C}$ a presentable stable category and $F: \mathcal{S} \xrightarrow{} \mathcal{C}$, let $\widetilde{F}$ be the functor $\widetilde{F}(X)=\mathrm{fib}( F(X) \xrightarrow{} F(*) ).$ For $M\in \mathcal{C}$, we write $C_*(-;M): \ \mathcal{S} \xrightarrow{} \mathcal{C}$ for $X\mapsto \mathbb{S}[X]\otimes M$.
\end{definition}

Recall that for any space $X$, there is an assembly map
\[ C_*(X;\tc(k)) \xrightarrow{} \tc(X;k) \]
in the convention of Definition \ref{tc-spc-def}. We will be interested in the cofibre of the assembly map.

\begin{definition}
    Let $\tc^W(-;k): \mathcal{S} \xrightarrow{} \Sp$ be the cofibre of assembly, so that there is a cofibre sequence
    \[ C_*(X;\tc(k)) \xrightarrow{} \tc(X;k) \xrightarrow{}\tc^W(X;k). \]
\end{definition}

Observe that the functors $\widetilde{F}$ and $\tc^W$ defined above both have the property that they are $0$ when evaluated on the point. The significance of such functors is that they are easily understood after pullback along $\cyclicorbit(C_p^r) \xrightarrow{} \mathcal{S}$, as we now explain.

\begin{lemma}\label{red-cyc-functor}
    For $\mathcal{C}$ a presentable stable category, let $\mathrm{Fun}^\mathrm{red}(\cyclicorbit(C_p^r), \mathcal{C})$ denote the full subcategory of functors with $F(G/e) \simeq 0$. Then for $\cyclicorbit^*(C_p^r)$ the orbit category on non-trivial cyclic subgroups, restriction along $\cyclicorbit^*(C_p^r) \hookrightarrow \cyclicorbit(C_p^r)$ induces an equivalence
    \[ \mathrm{Fun}^\mathrm{red}(\cyclicorbit(C_p^r), \mathcal{C}) \xrightarrow{\simeq} \mathrm{Fun}(\cyclicorbit^*(C_p^r), \mathcal{C}) \]
    with inverse given by extension by zero.
\end{lemma}
\begin{proof}
    This is immediate, since there are no maps $G/H \xrightarrow{} G/e$ for $H$ non-trivial.
\end{proof}

Moreover, any functor on $\cyclicorbit^*(C_p^r)$ that is restricted from a functor on spaces must in fact be constant.

\begin{lemma}\label{trivial-functor}
    The functor $\cyclicorbit^*(C_p^r) \xrightarrow{} \cyclicorbit(C_p^r) \xrightarrow{} \mathcal{S}$ factors through the point. Hence the image of the restriction $\mathrm{Fun}(\mathcal{S},\mathcal{C})\xrightarrow{}\mathrm{Fun}(\cyclicorbit^*(C_p^r),\mathcal{C})$ lies in the constant functors.
\end{lemma}
\begin{proof}
    The square
    \[\begin{tikzcd}
	{\cyclicorbit^*(C_p^r)} && {\cyclicorbit(C_p^r)} \\
	\\
	{*} && {\mathcal{S}}
	\arrow[from=1-1, to=1-3]
	\arrow[from=1-1, to=3-1]
	\arrow[from=1-3, to=3-3]
	\arrow["{* \mapsto BC_p}", from=3-1, to=3-3]
\end{tikzcd}\]
    commutes. Indeed, for $H$ a non-trivial cyclic subgroup of $G=C_p^r$, $(G/H)_{hG} \simeq BC_p$, so it suffices to see that the map
    \[ \text{Aut}_{\mathcal{O}(C_p^r)}(G/H) \xrightarrow{} \text{Aut}(BC_p) \]
    is trivial. But we can write $G=C_p \times C_p^{r-1}$ and $H=C_p\times *$ and factor this map as
    \[ \text{Aut}_{\mathcal{O}(C_p^r)}\big((C_p \times C_p^{r-1})/(C_p\times *)\big) \xrightarrow{(-)_{hC_p^{r-1}}} \text{Aut}_{\mathcal{O}(C_p)}(C_p/C_p) \xrightarrow{} \text{Aut}(BC_p) \]
    and $\text{Aut}_{\mathcal{O}(C_p)}(C_p/C_p) \simeq *$.
\end{proof}
\begin{remark}
    More generally, for $H \leq G$ the action of the Weyl group $W_G(H)$ on $G/H$ induces a trivial action of $W_G(H)$ on $BH$ precisely when the extension $H \to N_G(H) \to W_G(H)$ is isomorphic to the product $H \times W_G(H)$.
\end{remark}

Combining Lemma \ref{red-cyc-functor} and Lemma \ref{trivial-functor}, we deduce that when manipulating $\widetilde{F}$ and $\tc^W$ as functors on $\cyclicorbit(C_p^r)$, there are no equivariant considerations. 

We now recall the calculation of $\tc(k)$ for $k$ a perfect field of characteristic $p$. Let $W(k)$ be the ($p$-typical) Witt vectors of $k$ and $\phi: W(k) \xrightarrow{} W(k)$ the Frobenius. Then there is an exact sequence (\cite[Theorem B]{hesselholt1997k}, \cite{krause2018lectures}):

% https://q.uiver.app/#q=WzAsNixbMCwwLCIwIl0sWzEsMCwiXFx0Y18wKGspIl0sWzIsMCwiVyhrKSJdLFszLDAsIlcoaykiXSxbNCwwLCJcXHRjX3stMX0oaykiXSxbNSwwLCIwIl0sWzAsMV0sWzEsMl0sWzIsMywiXFxtYXRocm17SWR9LVxccGhpIl0sWzMsNF0sWzQsNV1d
\[\begin{tikzcd}
	0 & {\tc_0(k)} & {W(k)} & {W(k)} & {\tc_{-1}(k)} & 0
	\arrow[from=1-1, to=1-2]
	\arrow[from=1-2, to=1-3]
	\arrow["{\mathrm{Id}-\phi}", from=1-3, to=1-4]
	\arrow[from=1-4, to=1-5]
	\arrow[from=1-5, to=1-6]
\end{tikzcd}\]

We observe the following. For any $k$, $\tc_0(k) \simeq \mathbb{Z}_p$. For a finite field $k$, $\tc_{-1}(k)\simeq \mathbb{Z}_p$. For $k$ algebraically closed, $\tc_{-1}(k) \simeq 0$, so that $\tc(k) \simeq \mathbb{Z}_p$. Now for any $k$, the map $k \xrightarrow{} \overline{k}$ to the algebraic closure induces a map $\tc(k) \xrightarrow{} \tc(\overline{k}) \simeq \mathbb{Z}_p$ which splits the map $\mathbb{Z}_p \simeq \tc(k)_{\geq 0} \xrightarrow{} \tc(k)$. This gives a canonical splitting $\tc(k) \simeq \mathbb{Z}_p \oplus \Sigma^{-1}\tc_{-1}(k)$, which we will make use of.

Next, we recall several facts about the rank $1$ case $\tc(BC_p;k)$, which we will make use of. The algebra $k[C_p]$ is isomorphic to the truncated polynomial algebra $k[x]/(x^p)$, so by \cite{hesselholt1997cyclic} (see also \cite{speirs2020k}), 
\[ \tc_*(BC_p;k) \simeq \begin{cases}
    k^{n(p-1)} & *=2n-1, \ n \geq 1 \\
    \mathbb{Z}_p & *=0 \\
    \tc_{-1}(k) & *=-1 \\
    0 & \mathrm{otherwise}
\end{cases} \]
(the summand $\tc(k)$ of $\tc(BC_p;k)$, split off by the augmentation $k[C_p] \to k$, accounts for degrees $0$ and $-1$).

There is an assembly map for $\tc(BC_p;k)$, discussed in \cite{hesselholt2020topological}:
\[ C_*(BC_p;\tc(k)) \xrightarrow{} \tc(BC_p;k) \xrightarrow{} \tc^W(BC_p;k) \]
By \cite[Remark 1.4.8]{hesselholt2020topological}, the groups $\tc^W_*(BC_p;k)$ are given by
\[ \tc^W_*(BC_p;k) \simeq \begin{cases}
    k^{n(p-1)} & *=2n-1, \ n \geq 1 \\
    0 & \mathrm{otherwise}
\end{cases} \]
and so agree with those of $\tc_*(BC_p;k)$ in positive degrees, although the natural map is not an equivalence. (The vanishing in non-positive degrees can also be seen directly from the long exact sequence of assembly: on $\pi_{-1}$ the assembly map is the isomorphism $\tc_{-1}(k) \cong \tc_{-1}(BC_p;k)$ given by the inclusion of the summand $\tc(k)$, and on $\pi_0$ it is surjective, since it restricts to an isomorphism on the summand $\tc_0(k)\cong\mathbb{Z}_p$.)

We record two technical lemmas about the assembly map.
\begin{lemma}\label{BC_p-split-inj}
    The composite
    \[ C_*(BC_p; \mathbb{Z}_p) \xrightarrow{} C_*(BC_p;\tc(k)) \xrightarrow{} \tc(BC_p;k) \]
    where the first map is induced by $\mathbb{Z}_p \simeq \tc(k)_{\geq0} \xrightarrow{} \tc(k)$ is injective on homotopy. Moreover, the associated map $\widetilde{C}_*(BC_p;\mathbb{Z}_p) \xrightarrow{} \widetilde{\tc}(BC_p;k)$ admits a retraction.
\end{lemma}
\begin{proof}
    The first statement is clear on $\pi_0$. In even degrees, $H_*(BC_p;\mathbb{Z}_p)$ is $0$, so there is nothing to show. In odd degrees, we see from the long exact sequence of assembly that $H_*(BC_p;\tc(k)) \xrightarrow{} \tc_*(BC_p;k)$ is injective, since $\tc^W_*(BC_p;k)$ is $0$ in positive even degrees, and $H_*(BC_p;\mathbb{Z}_p) \xrightarrow{} H_*(BC_p;\tc(k))$ is (split) injective.

    For the second statement, observe that $\widetilde{C}_*(BC_p;\mathbb{Z}_p) \xrightarrow{} \widetilde{\tc}(BC_p;k)$ is a map of $\mathbb{Z}_p$-modules that is split injective on homotopy (since all homotopy groups have exponent $p$), from which it follows that the map is split injective.
\end{proof}

\begin{lemma}\label{BC_p-split-surj}
    The composite
    \[ \tc^W(BC_p;k) \xrightarrow{} \Sigma C_*(BC_p;\tc(k)) \xrightarrow{} \Sigma C_*(BC_p;\tc(k)_{<0})\simeq C_*(BC_p;\tc_{-1}(k))  \]
    is surjective on homotopy in positive degrees. Moreover, the associated map $\tc^W(BC_p;k) \xrightarrow{} \widetilde{C}_*(BC_p;\tc_{-1}(k))$ admits a section.
\end{lemma}
\begin{proof}
    The group $\tc_{-1}(k)$ is a torsion free $\mathbb{Z}_p$-module, so is flat. Hence $H_*(BC_p;\tc_{-1}(k))\cong H_*(BC_p;\mathbb{Z}_p)\otimes\tc_{-1}(k)$ is concentrated in degree $0$ and in odd degrees. The first statement then follows from the long exact sequence of assembly and the fact that $\tc_*(BC_p;k)$ is $0$ in even positive degrees.

    For the second statement, note that again $\tc^W(BC_p;k) \xrightarrow{} \widetilde{C}_*(BC_p;\tc_{-1}(k))$ is a map of $\mathbb{Z}_p$-modules, and is split surjective on homotopy (since all homotopy groups have exponent $p$), so that the map is split.
\end{proof}

\subsection{The assembly sequence for $\tc(BC_p^r;k)$}
We now analyse the assembly sequence
\[ C_*(BC_p^r;\tc(k)) \xrightarrow{} \tc(BC_p^r;k) \xrightarrow{} \tc^W(BC_p^r;k) \]
for $k$ a perfect field of characteristic $p$. We will show that the associated long exact sequence in homotopy reduces to short exact sequences, and that a certain key map is split injective. From this, we will deduce for $k$ finite an abstract isomorphism $\tc_n(BC_p^r;k) \cong \tc_n^W(BC_p^r;k)$ for $n>0$ (although the natural map will not be an isomorphism).

Recall that there is a canonical splitting $\tc(k) \simeq \mathbb{Z}_p \oplus \Sigma^{-1}\tc_{-1}(k)$. Let us first consider the composite
\[ C_*(-;\mathbb{Z}_p) = C_*(-;\tc(k)_{\geq 0}) \xrightarrow{} C_*(-;\tc(k)) \xrightarrow{} \tc(-;k) \]
applied to $BC_p^r$.

\begin{prop}\label{split-inj-ass}
    The above map induces a split injection on $H_n(BC_p^r;\mathbb{Z}_p) \xrightarrow{} \tc_n(BC_p^r;k)$ for all $n$.
\end{prop}
\begin{proof}
    We show that the map $C_*(BC_p^r;\mathbb{Z}_p) \xrightarrow{} \tc(BC_p^r;k)$ admits a retraction. The composite
    \[ \mathbb{Z}_p \xrightarrow{} \tc(k) \xrightarrow{} \tc(\overline{k}) \]
    is an equivalence, so by considering the diagram
    % https://q.uiver.app/#q=WzAsNSxbMCwwLCJDXyooQkNfcF5yO1xcbWF0aGJie1p9X3ApIl0sWzEsMCwiQ18qKEJDX3BecjtcXHRjKGspKSJdLFsyLDAsIkNfKihCQ19wXnI7XFx0YyhcXG92ZXJsaW5le2t9KSkiXSxbMSwxLCJcXHRjKEJDX3BecjtrKSJdLFsyLDEsIlxcdGMoQkNfcF5yO1xcb3ZlcmxpbmV7a30pIl0sWzAsMV0sWzEsM10sWzEsMl0sWzMsNF0sWzIsNF1d
\[\begin{tikzcd}
	{C_*(BC_p^r;\mathbb{Z}_p)} & {C_*(BC_p^r;\tc(k))} & {C_*(BC_p^r;\tc(\overline{k}))} \\
	& {\tc(BC_p^r;k)} & {\tc(BC_p^r;\overline{k})}
	\arrow[from=1-1, to=1-2]
	\arrow[from=1-2, to=1-3]
	\arrow[from=1-2, to=2-2]
	\arrow[from=1-3, to=2-3]
	\arrow[from=2-2, to=2-3]
\end{tikzcd}\]
    
    it is enough to show that $C_*(BC_p^r;\mathbb{Z}_p) \xrightarrow{} \tc(BC_p^r;\overline{k})$ admits a retraction. We may thus assume that $k$ is algebraically closed.
    
    Now $C_*(*;\mathbb{Z}_p)\simeq \tc(*;k)$, so there is a pushout of functors on $\cyclicorbit(C_p^r)$:
    % https://q.uiver.app/#q=WzAsNCxbMCwwLCJcXHdpZGV0aWxkZXtDfSgtO1xcbWF0aGJie1p9X3ApIl0sWzAsMSwiQygtO1xcbWF0aGJie1p9X3ApIl0sWzEsMSwiXFx0YygtO2spIl0sWzEsMCwiXFx3aWRldGlsZGV7XFx0Y30oLTtrKSJdLFswLDNdLFswLDFdLFsxLDJdLFszLDJdXQ==
\[\begin{tikzcd}
	{\widetilde{C}_*(-;\mathbb{Z}_p)} & {\widetilde{\tc}(-;k)} \\
	{C_*(-;\mathbb{Z}_p)} & {\tc(-;k)}
	\arrow[from=1-1, to=1-2]
	\arrow[from=1-1, to=2-1]
	\arrow[from=1-2, to=2-2]
	\arrow[from=2-1, to=2-2]
\end{tikzcd}\]
    We first show that the map $\widetilde{C}_*(-;\mathbb{Z}_p) \xrightarrow{} \widetilde{\tc}(-;k)$ admits a retraction. Both of these functors on $\cyclicorbit(C_p^r)$ are restricted from functors on $\mathcal{S}$, and vanish on the point, so by Lemma \ref{red-cyc-functor} and Lemma \ref{trivial-functor}, it suffices to show that the (non-equivariant) map $\widetilde{C}_*(BC_p;\mathbb{Z}_p) \xrightarrow{} \widetilde{\tc}(BC_p;k)$ admits a retraction, which is precisely Lemma \ref{BC_p-split-inj}. By standard properties of pushouts, the map $C_*(-;\mathbb{Z}_p) \xrightarrow{} \tc(-;k)$ then admits a retraction as functors on $\cyclicorbit(C_p^r)$.

    We conclude by taking colimits over $\cyclicorbit(C_p^r)$. Indeed, the cyclic assembly maps
    \begin{align*}
        \underset{\mathcal{O}_{\mathcal{F}_\mathrm{cyc}}(C_p^r)}{\colim} \ C_*(BH;\mathbb{Z}_p)&\xrightarrow{\simeq}  C_*(BC_p^r;\mathbb{Z}_p) \\
        \underset{\mathcal{O}_{\mathcal{F}_\mathrm{cyc}}(C_p^r)}{\colim} \ \tc(BH;k)&\xrightarrow{\simeq}  \tc(BC_p^r;k)
    \end{align*}
    are equivalences (the first since $C_*(-;\mathbb{Z}_p)$ preserves colimits and $\colim_{\mathcal{O}_\mathcal{F}(G)}BH \simeq (E\mathcal{F})_{hG} \simeq BG$, the second by Corollary \ref{cyc-induct-tc}), so we deduce that $C_*(BC_p^r;\mathbb{Z}_p) \xrightarrow{} \tc(BC_p^r;k)$ admits a retraction.
\end{proof}

Next, we consider the composite
\[ \tc^W(-;k) \xrightarrow{} \Sigma C_*(-;\tc(k)) \xrightarrow{} \Sigma C_*(-;\tc(k)_{<0}) \simeq C_*(-;\tc_{-1}(k))\]
where the first map is the boundary map for the assembly cofibre sequence.

\begin{prop}\label{bound-surj-ass}
    The above map induces surjections $\tc^W_n(BC_p^r;k) \xrightarrow{} H_n(BC_p^r;\tc_{-1}(k))$ for $n>0$.
\end{prop}
\begin{proof}
    Observe that $\tc^W(*;k)\simeq 0$, so the map $\tc^W(-;k) \xrightarrow{}  C_*(-;\tc_{-1}(k))$ factors as 
    \[ \tc^W(-;k) \xrightarrow{} \widetilde{C}_*(-;\tc_{-1}(k)) \xrightarrow{} C_*(-;\tc_{-1}(k)). \]
    
    We first show that the initial map admits a section as a functor on $\cyclicorbit(C_p^r)$. Indeed, by Lemma \ref{red-cyc-functor} and Lemma \ref{trivial-functor}, it suffices to show that $\tc^W(BC_p;k) \xrightarrow{} \widetilde{C}_*(BC_p;\tc_{-1}(k))$ admits a section, which is Lemma \ref{BC_p-split-surj}.

    By Corollary \ref{cyc-induct-tc}, the cyclic assembly maps for $C_*(-;\tc(k))$ and $\tc(-;k)$ are both equivalences, so the cyclic assembly map for $\tc^W(-;k)$ is also an equivalence, 
    \[  \underset{\mathcal{O}_{\mathcal{F}_\mathrm{cyc}}(C_p^r)}{\colim} \ \tc^W(BH;k)\xrightarrow{\simeq}  \tc^W(BC_p^r;k). \]
    Hence the map
    \[ \tc^W(BC_p^r;k) \xrightarrow{}  \underset{\mathcal{O}_{\mathcal{F}_\mathrm{cyc}}(C_p^r)}{\colim} \ \widetilde{C}_*(BH;\tc_{-1}(k)) \]
    admits a section, and in particular is surjective on homotopy. It remains to show that the assembly map
    \[ \underset{\mathcal{O}_{\mathcal{F}_\mathrm{cyc}}(C_p^r)}{\colim} \ \widetilde{C}_*(BH;\tc_{-1}(k)) \xrightarrow{} \widetilde{C}_*(BC_p^r;\tc_{-1}(k)) \]
    is surjective on homotopy, which we prove as Lemma \ref{group-homology-surj} below.
    
\end{proof}

\begin{lemma}\label{group-homology-surj}
    For any abelian group $M$, the assembly map
    \[ \underset{\mathcal{O}_{\mathcal{F}_\mathrm{cyc}}(C_p^r)}{\colim} \ \widetilde{C}_*(BH;M) \xrightarrow{} \widetilde{C}_*(BC_p^r;M) \]
    is surjective on homotopy.
\end{lemma}
\begin{proof}
     We first show the case $M=\mathbb{Z}$. Observe that since $\widetilde{C}_*(*) \simeq 0$, the colimit over $\mathcal{O}_{\mathcal{F}_\mathrm{cyc}}(C_p^r)$ evaluates as
    \[ \underset{\mathcal{O}_{\mathcal{F}_\mathrm{cyc}}(C_p^r)}{\colim} \ \widetilde{C}_*(BH) \simeq \bigoplus_{|H|=p}\widetilde{C}_*(BH)_{hC_p^r/H}.\]
    Now, by definition for each $H$ there is an (equivariant) cofibre sequence
    \[ \widetilde{C}_*(BH) \xrightarrow{} C_*(BH) \xrightarrow{} C_*(*) \]
    and so cofibre sequences
    \[ \widetilde{C}_*(BH)_{hC_p^r/H} \xrightarrow{} C_*(BC_p^r) \xrightarrow{} C_*(B(C_p^r/H)). \]
    The map $C_p^r \xrightarrow{} C_p^r/H$ is split, hence $C_*(BC_p^r) \xrightarrow{} C_*(B(C_p^r/H))$ is surjective on homotopy, and so we may identify the image of $\pi_k(\widetilde{C}_*(BH)_{hC_p^r/H}) \xrightarrow{} H_k(BC_p^r)$ with the kernel of $H_k(BC_p^r) \xrightarrow{} H_k(B(C_p^r/H))$.

    We must therefore show the following elementary fact: the submodules $\mathrm{ker}(H_k(BC_p^r) \xrightarrow{} H_k(B(C_p^r/H)))$, for $H$ the cyclic subgroups, generate $H_k(BC_p^r)$ for $k>0$. To see this, first note that all groups have exponent $p$, so we may instead show the $\mathbb{F}_p$-dual statement. That is, we will show that the intersection of the subspaces $\mathrm{im}(H_k(B(C_p^r/H))^\vee \xrightarrow{} H_k(BC_p^r)^\vee)$ is $0$, where $M^\vee = \mathrm{Hom}(M, \mathbb{F}_p)$.

    We have natural isomorphisms $H_{k}(BG)^\vee\cong \mathrm{Ext}(H_{k}(BG), \mathbb{Z})\cong H^{k+1}(BG)$, so we must show that the intersection of $\mathrm{im}(H^{k+1}(B(C_p^r/H)) \xrightarrow{} H^{k+1}(BC_p^r))$ is $0$. The integral cohomology of these groups injects into $\mathbb{F}_p$-cohomology, so it suffices to show the corresponding statement with $\mathbb{F}_p$-coefficients. But this is clear from the well known ring structure in this case.

    Now suppose $M$ is a general abelian group, and set
    \begin{align*}
        X &= \underset{\mathcal{O}_{\mathcal{F}_\mathrm{cyc}}(C_p^r)}{\colim} \ \widetilde{C}_*(BH;\mathbb{Z}) \\
        Y &= \widetilde{C}_*(BC_p^r;\mathbb{Z})
    \end{align*}

    We must show that $X \otimes M \xrightarrow{} Y \otimes M$ is surjective on homotopy. From the integral case, $\pi_n(X) \xrightarrow{} \pi_n(Y)$ is surjective, and necessarily split since $\pi_n(X), \pi_n(Y)$ are $\mathbb{F}_p$-modules. Then the result for $X \otimes M \xrightarrow{} Y \otimes M$ follows from comparing the universal coefficient sequences.
\end{proof}

We now have all the necessary tools to analyse the long exact sequence in homotopy of the assembly map. The long exact sequence has the form:
\[ \ldots \xrightarrow{} H_n(BC_p^r;\mathbb{Z}_p) \oplus H_{n+1}(BC_p^r;\tc_{-1}(k)) \xrightarrow{} \tc_n(BC_p^r;k) \xrightarrow{} \tc_n^W(BC_p^r;k) \xrightarrow{} \ldots \]
By Proposition \ref{split-inj-ass}, $H_n(BC_p^r;\mathbb{Z}_p) \xrightarrow{} \tc_n(BC_p^r;k)$ is split injective. By Proposition \ref{bound-surj-ass}, the boundary map $\tc_n^W(BC_p^r;k) \xrightarrow{} H_n(BC_p^r;\tc_{-1}(k))$ is surjective. Hence for each $n>0$, the long exact sequence decomposes into short exact sequences
\[ 0 \xrightarrow{} H_n(BC_p^r;\mathbb{Z}_p) \xrightarrow{} \tc_n(BC_p^r;k) \xrightarrow{} \tc_n^W(BC_p^r;k) \xrightarrow{} H_n(BC_p^r;\tc_{-1}(k)) \xrightarrow{} 0 \]
in which the first map is split.

\begin{prop}\label{k-fin-ass-iso}
    Suppose that $k$ is finite, and $H \leq \mathrm{GL}_r(\mathbb{F}_p)$ has order prime to $p$. Then there is an isomorphism of $H$-modules
    \[ \tc_n(BC_p^r;k) \cong \tc_n^W(BC_p^r;k) \]
    for $n>0$.
\end{prop}
\begin{proof}
    In this case, $\tc_{-1}(k) \cong \mathbb{Z}_p$, so that there is an exact sequence of $H$-modules
    \[ 0 \xrightarrow{} H_n(BC_p^r;\mathbb{Z}_p) \xrightarrow{} \tc_n(BC_p^r;k) \xrightarrow{} \tc_n^W(BC_p^r;k) \xrightarrow{} H_n(BC_p^r;\mathbb{Z}_p) \xrightarrow{} 0 \]
    We show in Proposition \ref{tcw-calc} below that $\tc_n^W(BC_p^r;k)$ is a finite $\mathbb{F}_p$-module. The finite groups $H_n(BC_p^r;\mathbb{Z}_p)$ have exponent $p$, and the initial map is split, so the above sequence is a sequence of finite $\mathbb{F}_p[H]$-modules. By semi-simplicity, an $\mathbb{F}_p[H]$-module is determined by its class in $K_0$, and we have that
    \[ [\tc_n(BC_p^r;k)] = [ H_n(BC_p^r;\mathbb{Z}_p)] + [\tc_n^W(BC_p^r;k)] - [H_n(BC_p^r;\mathbb{Z}_p)] = [\tc_n^W(BC_p^r;k)] \in K_0(\mathbb{F}_p[H]) \]
    so $\tc_n(BC_p^r;k) \cong \tc_n^W(BC_p^r;k)$.
\end{proof}

\subsection{Conclusion}
We conclude by calculating the groups $\tc^W_n(BC_p^r;k)$, and describing the action of certain subgroups $H \leq \mathrm{GL}_r(\mathbb{F}_p)$.

\begin{prop}\label{tcw-calc}
    For $k$ a perfect field of characteristic $p$, 
    \[ \tc_i^W(BC_p^r;k) \cong k^{n_i} \]
    where the $n_i$ are given by the Hilbert series
    \[ (p^r-1)(1+x+x^2+\dots)^{r-1}(x+2x^3+3x^5+\dots). \]
\end{prop}
\begin{proof}
    As in the proof of Proposition \ref{bound-surj-ass}, the cyclic assembly map is an equivalence
    \[  \underset{\mathcal{O}_{\mathcal{F}_\mathrm{cyc}}(C_p^r)}{\colim} \ \tc^W(BH;k)\xrightarrow{\simeq}  \tc^W(BC_p^r;k). \]
    Now by Lemma \ref{red-cyc-functor} and Lemma \ref{trivial-functor}, $\tc^W(-;k)$ is constant as a functor on $\cyclicorbit^*(C_p^r)$, so we find that
    \[  \underset{\mathcal{O}_{\mathcal{F}_\mathrm{cyc}}(C_p^r)}{\colim} \ \tc^W(BH;k) \simeq  \underset{\cyclicorbit^*(C_p^r)}{\colim} \ \tc^W(BH;k) \simeq \bigoplus_{|H|=p}\mathbb{S}[BC_p^{r-1}]\otimes \tc^W(BC_p;k).  \]
    Now by observation of the homotopy groups (and using that $\mathbb{Z}_p$-modules split as the sum of their homotopy groups) $\tc^W(BC_p;k)$ is a $k$-module, so that 
    \[\tc^W(BC_p^r;k) \simeq \bigoplus_{|H|=p}k[BC_p^{r-1}]\otimes_k \tc^W(BC_p;k). \]
    There are $(p^r-1)/(p-1)$ subgroups of order $p$, and over $k$ the homotopy groups of $k[BC_p^{r-1}]$ and of $\tc^W(BC_p;k)$ have Poincar\'e series $(1+x+x^2+\dots)^{r-1}$ and $(p-1)(x+2x^3+3x^5+\dots)$ respectively, which gives the stated formula.
\end{proof}

The previous proposition shows in particular that the groups $\tc_i^W(BC_p^r;k)$ are $p$-torsion, so have the structure of $\mathbb{F}_p[\mathrm{Out}(C_p^r)]\cong \mathbb{F}_p[\mathrm{GL}_r(\mathbb{F}_p)]$ modules, which are finite dimensional when $k$ is finite. While we cannot describe the structure explicitly, we show that upon restricting to certain subgroups these representations become free. The class of subgroups we consider are those $H \leq \mathrm{GL}_r(\mathbb{F}_p)$ with the property that $H$ acts freely on $\mathbb{F}_p^r \setminus \{0\}$. The most notable example is the Singer cycle $C_{p^r-1} \cong \mathbb{F}_{p^r}^\times \leq  \mathrm{GL}_r(\mathbb{F}_p)$, but there are also non-cyclic examples, such as the irreducible $2$-dimensional representation $Q_8 \leq \mathrm{GL}_2(\mathbb{F}_3)$.

\begin{prop}\label{free-action}
    Let $k$ be finite. Suppose that $H \leq \mathrm{GL}_r(\mathbb{F}_p)$ acts freely on $\mathbb{F}_p^r \setminus \{0\}$. Then the action of $H$ on $\tc_i^W(BC_p^r;k)$ is free.
\end{prop}
\begin{proof}
    Let $T$ be the set of non-trivial elements of $C_p^r$. The free loop space $LBC_p^r$ decomposes (as a space with $S^1$-action)
    \[ LBC_p^r \simeq BC_p^r \bigsqcup L_TBC_p^r \]
    in the notation of Definition \ref{LTBG-def}. Recall that, since $\{1\}$ is closed under $x \mapsto x^p$, $BC_p^r$ is again a space with Frobenius lifts, and there is a map of cyclotomic spectra
    \[ \thh(k) \otimes \mathbb{S}[BC_p^r] \xrightarrow{}  \thh(k) \otimes \mathbb{S}[LBC_p^r] \]
    and so a map 
    \[ \tc \big(\thh(k) \otimes \mathbb{S}[BC_p^r];p \big)\xrightarrow{}  \tc \big( \thh(k) \otimes \mathbb{S}[LBC_p^r];p \big)=\tc(BC_p^r;k).\]
    Moreover, using again that by \cite[Theorem 2.7]{clausen2021}, the functor $\tc(-)/p$ commutes with colimits of connective cyclotomic spectra, the above map can be identified with the assembly map. Now by Proposition \ref{tc-ss}, we have an equivariant identification
    \[ \tc^W(BC_p^r;k) \simeq \big(\Sigma (\thh(k) \otimes \mathbb{S}[L_TBC_p^r])_{hS^1}\big)^\wedge_p.  \]
    As a space with $S^1$-action, we have that
    \[ L_TBC_p^r \simeq \bigsqcup_{g \neq e} L_{\{g\}}BC_p^r  \]
    and by assumption, the action of $H$ permutes the components freely.
    Hence there is a decomposition of $\mathbb{F}_p$-modules (finite dimensional by Proposition \ref{tcw-calc})
    \[ \tc_i^W(BC_p^r;k) \cong \bigoplus_{g \neq e}\pi_i\big(\Sigma (\thh(k) \otimes \mathbb{S}[L_{\{g\}}BC_p^r])_{hS^1}\big)^\wedge_p \]
    with the following property: for $\phi \in H$, and $g \neq e$, the summand corresponding to $g$ is mapped to the summand corresponding to $\phi(g)$. But this property implies that the $H$-representation is induced from the trivial subgroup, hence free.
    
\end{proof}

We may now finish our calculation of $\tc_i(k[C_p^r];p)$.

\begin{theorem}\label{tc-el-al}
    Let $k$ be a finite field of characteristic $p$. For $i>0$,
    \[\tc_i(k[C_p^r];p) = \tc_i(BC_p^r;k) \cong k^{n_i} \]
    where the dimensions $n_i$ are given by the coefficients in the Hilbert series
    \[ (p^r-1)(1+x+x^2+\dots)^{r-1}(x+2x^3+3x^5+\dots). \]

    Moreover, if $H \leq \mathrm{GL}_r(\mathbb{F}_p)$ acts freely on $\mathbb{F}_p^r \setminus \{0\}$, then the action of $H$ on $\tc_i(BC_p^r;k)$ is free.
\end{theorem}
\begin{proof}
    Let $H$ be as in the statement (possibly trivial). The condition on $H$ implies that the order of $H$ is prime to $p$, so Proposition \ref{k-fin-ass-iso} states that there is an isomorphism of $\mathbb{F}_p[H]$-modules $\tc_i(BC_p^r;k) \cong \tc^W_i(BC_p^r;k)$. The latter has the required dimension formula by Proposition \ref{tcw-calc}, and is free by Proposition \ref{free-action}.
\end{proof}
\section{Groups of Lie type $A_1$}
We now apply the results of the previous sections to calculate $\tc_i(k[G];p)$ for certain finite groups of Lie type in equal characteristic. We will make the following (non-standard) definition.

\begin{definition}
    The finite groups of Lie type $A_1$ are the families $\psl_2(\mathbb{F}_q)$, $\pgl_2(\mathbb{F}_q)$, $\spl_2(\mathbb{F}_q)$, and $\gl_2(\mathbb{F}_q)$.
\end{definition}

We will consider coefficients with the same characteristic as $\mathbb{F}_q$. 

\subsection{Reduction to the Borel subgroup}
We first show that we may reduce the calculation to the calculation for $k[N]$, where $N$ is a normaliser of a Sylow $p$-subgroup. For the groups of Lie type $A_1$, such a subgroup $N$ is a Borel subgroup of upper triangular matrices (equivalently, the stabiliser of a flag). We prove this reduction using the stable module category. Note that under the further assumption that $k$ contains enough roots of unity, this reduction can also be conceptually framed via $G$-spectra and Brown's simplicial complex; we provide a full account of this perspective in Appendix A. The use of this stable module category for $K$-theory calculations is also explored in \cite{vogeli2025derived}; see there for more details and applications.

\begin{definition}
    For a finite group $G$ and a field $k$, the stable module category of $k[G]$ is the Verdier quotient
    \[ \stMod(k[G]) = D^\mathrm{b}(k[G])/D^\mathrm{perf}(k[G]) \]
    where $D^\mathrm{b}(k[G])$ is the derived category of bounded, finite dimensional $k[G]$-complexes.
\end{definition}

By the localisation property of $K$-theory, we have a cofiber sequence
\[ K(k[G]) \xrightarrow{} K(D^\mathrm{b}(k[G])) \xrightarrow{} K(\stMod(k[G])). \]
The middle term is simple to understand. Indeed, $D^\mathrm{b}(k[G])$ is generated as a thick subcategory by the simple modules $S_1,\ldots,S_n$ of $k[G]$, and so a dévissage argument (\cite[Proposition 3.2.1]{vogeli2025derived}) implies that $K(D^\mathrm{b}(k[G]))\simeq \bigoplus K(\mathrm{End}(S_i))$. The latter is equivalent to $K(k[G]/J)$, and $K_i(k[G]/J)$ is uniquely $p$-divisible for $i>0$ (see Lemma \ref{artin-wed}). In particular, $K_{i+1}(\stMod(k[G]))$ and $K_i(k[G])$ share the same $p$-torsion for $i>0$.

Now suppose that $H \xrightarrow{} G$ induces an equivalence $\stMod(k[H]) \xrightarrow{} \stMod(k[G])$. Then from the long exact sequence of the above cofibre sequence, it follows that $K_i(k[H]) \xrightarrow{} K_i(k[G])$ induces an isomorphism on the $p$-torsion subgroup for $i>0$. But (by Lemma \ref{k-split} and the surrounding discussion), this is equivalent to $\tc_i(k[H];p) \xrightarrow{} \tc_i(k[G];p)$ being an isomorphism for $i>0$.

\begin{prop}\label{red-to-borel}
    Suppose that $G$ is a finite group of Lie type $A_1$ in characteristic $p$, and $N$ is the normaliser of a Sylow $p$-subgroup. Then for $k$ a perfect field of characteristic $p$, the maps
    \[ \tc_i(k[N];p) \xrightarrow{} \tc_i(k[G];p) \]
    are isomorphisms for $i>0$.
\end{prop}
\begin{proof}
    By the preceding discussion, it suffices to show that the induction map $\stMod(k[N]) \xrightarrow{} \stMod(k[G])$ is an equivalence of categories: this is standard, but we show it here for completeness. 

    Consider the composite of induction and restriction on $\stMod(k[N])$. By the Mackey formula, this sends a module $M$ to
    \[ \bigoplus_{g \in N\backslash G/ N} \mathrm{ind}^N_{N \cap^gN} (\mathrm{res}^{^gN}_{N \cap ^gN} \ ^gM). \]
    For $g \notin N$, the intersection $N \cap^gN$ has order prime to $p$ (it is conjugate to a subgroup of the diagonal torus). Hence the corresponding factors in the direct sum are projective, so they vanish in the stable module category, and the above expression is simply $M$.

    For the other direction, note that by the projection formula, restriction followed by induction is equivalent to tensoring with $k[G/N]$ in $\stMod(k[G])$. Hence we must show that the natural map $k[G/N] \xrightarrow{} k$ is an equivalence, or, equivalently, that the kernel in $k[G]$-modules is projective. For this it suffices to check that the kernel is projective after restricting to $N$, since $N$ contains a Sylow subgroup. But this holds by the previous paragraph.
\end{proof}

\subsection{$\tc$ of the Borel subgroup}
We now describe $\tc_i(k[N];p)$, where $N$ is again the Borel subgroup (the normaliser of a Sylow), and $k$ is a finite field of characteristic $p$. For the families $\psl$ and $\pgl$, the proof is immediate from the work we have done, since in these cases $N$ is a Frobenius group.

\begin{prop}\label{psl-pgl-cal}
    Suppose that $G$ is either $\psl_2(\mathbb{F}_q)$ with $q$ odd, or $\pgl_2(\mathbb{F}_q)$, where $q=p^r$. Let $N$ be the normaliser of a Sylow $p$-subgroup, and $k$ a finite field of characteristic $p$. Then for $i>0$,
    \[ \tc_i(k[N];p) \cong k^{n_i} \]
    where the $n_i$ are given by the Hilbert series:
    \[     \begin{cases}
        2(1+x+x^2+\dots)^{r-1}(x+2x^3+3x^5+\dots) & G= \psl_2(\mathbb{F}_q), \ q \text{ odd} \\
        (1+x+x^2+\dots)^{r-1}(x+2x^3+3x^5+\dots) & G= \pgl_2(\mathbb{F}_q)
    \end{cases}\]

\end{prop}
\begin{proof}
    In each case, the group $N$ is a Frobenius group of the form $N \cong C_p^r \rtimes H$, where $H$ has order prime to $p$. By Proposition \ref{tc-frob-groups}, we then have that 
    \[ \tc_i(k[N];p) \cong \tc_i(k[C_p^r];p)_H. \]
    We evaluate this directly from Theorem \ref{tc-el-al}. The action of $H$ on $C_p^r$ is of the form required in Theorem \ref{tc-el-al}, so the action on $\tc_i(k[C_p^r];p)$ is free, and we have that $\tc_i(k[C_p^r];p)_H$ is an $\mathbb{F}_p$-module with dimension
    \[ \mathrm{dim}_{\mathbb{F}_p}\big(\tc_i(k[C_p^r];p)\big)/|H| \]
    which gives the required formulas, since $|H|=(q-1)/2$ for $\psl_2(\mathbb{F}_q)$ and $|H|=q-1$ for $\pgl_2(\mathbb{F}_q)$.
\end{proof}

We now turn to Borel subgroups for the families $\spl$ and $\gl$. In these cases, the Borel subgroup is not Frobenius. Nevertheless, we show that one may in fact reduce the calculation to the Borel subgroups of $\psl$ and $\pgl$. 

If $k$ is a splitting field, there is a slick argument for the reduction: we sketch this for $\gl$. For $N$ the Borel subgroup of $\gl_2(\mathbb{F}_q)$, and $k$ a splitting field (equivalently, an extension of $\mathbb{F}_q$), the group algebra decomposes as
\[ k[N] \cong B_1\oplus\ldots\oplus B_{q-1} \]
(such $B_i$ are known as blocks of $k[N]$). Moreover each algebra $B_i$ is isomorphic to the group algebra of the Borel subgroup of $\pgl_2(\mathbb{F}_q)$. Thus we may deduce the groups $\tc_i(k[N];p)$ from Proposition \ref{psl-pgl-cal}. 

When $k$ is not a splitting field, the above argument cannot work, since the idempotents defining the block decomposition are not defined. We instead analyse the spaces with Frobenius lift $LBN$. It is interesting to observe that such an analysis can be used in place of block theory.

We first deal with the simpler case of $\spl$.

\begin{prop}\label{sl-cal}
    Suppose that $G$ is $\spl_2(\mathbb{F}_q)$, where $q=p^r$ for $p$ odd. Let $N$ be the normaliser of a Sylow $p$-subgroup, and $k$ a finite field of characteristic $p$. Then for $i>0$,
    \[ \tc_i(k[N];p) \cong k^{n_i} \]
    where the $n_i$ are given by the Hilbert series:
    \[ 4(1+x+x^2+\dots)^{r-1}(x+2x^3+3x^5+\dots) \]
\end{prop}
\begin{proof}
    Let $K \leq N$ be the subgroup of elements with equal diagonal values,
    \[ K =  \left\{ \begin{pmatrix} \lambda & x\\ 0 & \lambda \end{pmatrix} : \lambda = \pm 1, \ x\in \mathbb{F}_q\right\}\]
    and $T$ the set of elements with differing diagonal values,
    \[T = \left\{ \begin{pmatrix} \lambda & x\\ 0 & \lambda^{-1} \end{pmatrix} : \lambda\in \mathbb{F}_q^\times \setminus \{1,-1\}, \ x \in \mathbb{F}_q\right\} \]
    so that $N = K \sqcup T$.

    The sets $K, T$ are closed under conjugation and taking $p^\mathrm{th}$ powers, so by Proposition \ref{sub-frob-lift}, there is a decomposition of spaces with $p$-Frobenius lifts
    \[ LBN \simeq L_{K}BN \sqcup L_{T}BN \]
    and so an associated decomposition of $p$-cyclotomic spectra
    \[ \thh(k[N])\simeq \thh(k) \otimes \mathbb{S}[LBN]\simeq \thh(k)\otimes \big( \mathbb{S}[L_KBN] \oplus \mathbb{S}[L_TBN]\big). \]

    We first show that $\tc_i\big(\thh(k)\otimes \mathbb{S}[L_TBN] ;p\big)$ is $0$ for $i>0$. For this, let $D$ denote the subgroup of diagonal matrices. Observe that the $D$-conjugacy classes of $T \cap D$ are in bijection with the $N$-conjugacy classes of $T$, and that for $g \in T \cap D$ the inclusion $C_D(g) \hookrightarrow{} C_N(g)$ is an isomorphism, so that
    \[ L_{T\cap D}BD \xrightarrow{} L_TBN \]
    is an equivalence of spaces with $p$-Frobenius lifts. Now
    \[ LBD \simeq L_{K\cap D}BD \sqcup L_{T \cap D}BD \]
    as spaces with $p$-Frobenius lifts, hence the $p$-cyclotomic spectrum
    \[ \thh(k) \otimes \mathbb{S}[L_TBN] \simeq \thh(k) \otimes \mathbb{S}[L_{T\cap D}BD] \]
    is a summand of $\thh(k) \otimes \mathbb{S}[LBD] \simeq \thh(k[D])$. But $k[D]$ is a perfect $\mathbb{F}_p$-algebra, so $\tc_i(k[D];p)\cong 0$ for $i>0$, hence $\tc_i\big(\thh(k)\otimes \mathbb{S}[L_TBN] ;p\big)\cong 0$ for $i>0$.

    It remains to calculate the groups $\tc_i\big(\thh(k)\otimes \mathbb{S}[L_KBN] ;p\big)$. For this, we use Proposition \ref{tc-normal}, which states that there is an equivalence
    \[ \tc(k[K];p)_{h\mathbb{F}_q^\times/\{\pm 1 \}} \xrightarrow{\simeq} \tc\big(\thh(k)\otimes \mathbb{S}[L_KBN] ;p\big). \]

    Now $\mathbb{F}_q^\times/\{\pm 1\}$ has order prime to $p$, so the homotopy groups of the left hand side are simply given by the orbits $\tc_i(k[K];p)_{\mathbb{F}_q^\times/\{\pm 1\}}$. Moreover, the action of $\mathbb{F}_q^\times/\{\pm 1\}$ on the abelian group $\tc_i(k[K];p)$ depends only on the map $\mathbb{F}_q^\times/\{\pm 1\} \xrightarrow{} \mathrm{Out}(K)$, hence can be computed via the action on the algebra $k[K]$ (this is not generally true on the level of spectra).

    On the group $K \cong C_2 \times \mathbb{F}_q$, the class of $z \in \mathbb{F}_q^\times$ in $\mathbb{F}_q^\times/\{\pm 1\}$ acts trivially on the first factor, and via $x \mapsto z^2x$ on the second. Hence we see that there is an isomorphism of algebras with $\mathbb{F}_q^\times/\{\pm 1\}$-action
    \[ k[K] \simeq k[\mathbb{F}_q] \oplus k[\mathbb{F}_q]. \]
    We then compute
    \[ \tc_i(k[K];p)_{\mathbb{F}_q^\times/\{\pm 1\}} \cong \tc_i(k[\mathbb{F}_q];p)_{\mathbb{F}_q^\times/\{\pm 1\}} \oplus \tc_i(k[\mathbb{F}_q];p)_{\mathbb{F}_q^\times/\{\pm 1\}}  \]
    which gives the stated formula via Theorem \ref{tc-el-al}.
\end{proof}

We now treat the case of $\gl$. The proof is much the same, with the key difference that the situation becomes slightly more complex when $k$ is not a splitting field.

\begin{prop}\label{gl-calc}
    Suppose that $G$ is $\gl_2(\mathbb{F}_q)$, where $q=p^r$. Let $N$ be the normaliser of a Sylow $p$-subgroup, and $k$ a finite field of characteristic $p$. Then for $i>0$,
    \[ \tc_i(k[N];p) \cong k^{n_i} \]
    where the $n_i$ are given by the Hilbert series:
    \[ (q-1)(1+x+x^2+\dots)^{r-1}(x+2x^3+3x^5+\dots) \]
\end{prop}
\begin{proof}
    As in the previous proof, let $K \leq N$ be the subgroup of elements with equal diagonal values,
    \[ K =  \left\{ \begin{pmatrix} \lambda & x\\ 0 & \lambda \end{pmatrix} : \lambda \in \mathbb{F}_q^\times, \ x\in \mathbb{F}_q\right\}\]
    and $T$ the set of elements with differing diagonal values,
    \[T = \left\{ \begin{pmatrix} \lambda & x\\ 0 & \mu \end{pmatrix} : \lambda,\mu\in \mathbb{F}_q^\times, \ \lambda \neq \mu, \ x \in \mathbb{F}_q\right\} \]
    so that $N = K \sqcup T$.

    We have decompositions 
    \[ LBN \simeq L_{K}BN \sqcup L_{T}BN \]
    and
    \[ \thh(k[N])\simeq \thh(k) \otimes \mathbb{S}[LBN]\simeq \thh(k)\otimes \big( \mathbb{S}[L_KBN] \oplus \mathbb{S}[L_TBN]\big). \]
    We first show that $\tc_i\big(\thh(k)\otimes \mathbb{S}[L_TBN] ;p\big)$ is $0$ for $i>0$. Indeed, letting $D$ denote the subgroup of diagonal matrices, we have by the same argument as in the previous proof that $\tc_i\big(\thh(k)\otimes \mathbb{S}[L_TBN] ;p\big)$ is a summand of $\tc_i(k[D];p)$, which is $0$ since $k[D]$ is a perfect $\mathbb{F}_p$-algebra.

    We now turn to $\tc_i\big(\thh(k)\otimes \mathbb{S}[L_KBN];p\big)$. By Proposition \ref{tc-normal}, there is an equivalence
    \[ \tc(k[K];p)_{h\mathbb{F}_q^\times} \xrightarrow{\simeq} \tc\big(\thh(k)\otimes \mathbb{S}[L_KBN] ;p\big). \]
    Now $\mathbb{F}_q^\times$ acts on $K \cong C_{q-1}\times \mathbb{F}_q$ trivially on the first factor, and by multiplication on the second. Consider the ring $k[C_{q-1}]$. By semi-simplicity, it splits as a product of finite field extensions of $k$, 
    \[ k[C_{q-1}] \cong k_1 \oplus \ldots \oplus k_l \]
    Now we obtain an isomorphism of algebras with $\mathbb{F}_q^\times$-action:
    \[ k[K]\cong k[C_{q-1}]\otimes_k k[\mathbb{F}_q] \cong k_1[\mathbb{F}_q]\oplus \ldots \oplus k_l[\mathbb{F}_q]. \]
    As in the proof of Proposition \ref{sl-cal}, the action of $\mathbb{F}_q^\times$ on the group $\tc_i(k[K];p)$ may equivalently be described via the action on the algebra $k[K]$. Hence we find that 
    \[ \pi_i(\tc(k[K];p)_{h\mathbb{F}_q^\times}) \cong \tc_i(k[K];p)_{\mathbb{F}_q^\times} \cong \tc_i(k_1[\mathbb{F}_q];p)_{\mathbb{F}_q^\times} \oplus \ldots \oplus \tc_i(k_l[\mathbb{F}_q];p)_{\mathbb{F}_q^\times}. \]
    Now applying Theorem \ref{tc-el-al} to the finite fields $k_j$, we evaluate the above as
    \[\tc_i(k_j[\mathbb{F}_q];p)_{\mathbb{F}_q^\times} \cong k_j^{m_i} \]
    where the $m_i$ are given by the Hilbert series
    \[(1+x+x^2+\dots)^{r-1}(x+2x^3+3x^5+\dots).\]
    But then
    \[ \tc_i(k[K];p)_{\mathbb{F}_q^\times}\cong k_1^{m_i}\oplus\ldots\oplus k_l^{m_i} \cong (k[C_{q-1}])^{m_i} \cong k^{n_i}\]
    where the $n_i$ are as in the statement.
\end{proof}

We may now conclude our calculation.

\begin{theorem}\label{main-thm}
    Suppose that $G$ is a finite group of Lie type $A_1$ over $\mathbb{F}_q$, where $q=p^r$, and $k$ is a finite field of characteristic $p$. Then for $i>0$
    \[\tc_i(k[G];p) \cong k^{n_i} \]
    where the $n_i$ are given by a Hilbert series of the form
    \[ C(1+x+x^2+\dots)^{r-1}(x+2x^3+3x^5+\dots). \]
    The constant $C$ has values
    \begin{itemize}
        \renewcommand\labelitemi{--}
        \item $C=2$ for $G=\psl_2(\mathbb{F}_q)$ with $q$ odd,
        \item $C=1$ for $G=\pgl_2(\mathbb{F}_q)$,
        \item $C=4$ for $G=\spl_2(\mathbb{F}_q)$ with $q$ odd,
        \item $C=q-1$ for $G=\gl_2(\mathbb{F}_q)$.
    \end{itemize}
    For $q$ even, $\psl_2(\mathbb{F}_q)=\spl_2(\mathbb{F}_q)\cong\pgl_2(\mathbb{F}_q)$, so that $C=1$ in these cases.
\end{theorem}
\begin{proof}
    We first use that in each case $\tc_i(k[G];p)\cong \tc_i(k[N];p)$ by Proposition \ref{red-to-borel}, where $N$ is a Borel subgroup (normaliser of a Sylow $p$-subgroup). Then the calculations are Proposition \ref{psl-pgl-cal}, Proposition \ref{sl-cal}, and Proposition \ref{gl-calc}.
\end{proof}

\subsection{From $\tc$ to $K$-theory}
We conclude with a brief guide on how $\tc$ calculations relate to algebraic $K$-theory for group rings, and in particular how to deduce Theorem \ref{intro-thm-a} of the introduction. 

Recall that the Jacobson radical of $k[G]$ is the nilpotent ideal $J$ of elements that act by $0$ on all simple modules. The quotient $k[G]/J$ is the semisimplification of $k[G]$. The $K$-theory of $k[G]$ splits into direct sum of that of the semisimple algebra $k[G]/J$ and a relative term.

\begin{lemma}\label{k-split}
    Let $k$ be a perfect field of characteristic $p$. Then for a finite group $G$ and $J$ the Jacobson radical, there is an equivalence 
    \[ K(k[G])\simeq K(k[G]/J) \oplus K(k[G],J). \]
\end{lemma}
\begin{proof}
    By the Wedderburn-Malcev theorem, the map $k[G] \xrightarrow{} k[G]/J$ admits an algebra splitting, so $K(k[G]) \xrightarrow{} K(k[G]/J)$ is split.
\end{proof}

The factor $K(k[G],J)$ is entirely controlled by $\tc$, as we now explain. In our investigation, this factor is of primary interest, since it relates most to the group structure.

\begin{prop}
    For $k$, $G$ as above, we have that 
    \[ K(k[G],J) \simeq \tc(k[G],J). \]
    Moreover, for $i>0$, $\tc_i(k[G],J) \cong \tc_i(k[G];p)$.
\end{prop}
\begin{proof}
    The first statement is the Dundas-Goodwillie-McCarthy theorem (\cite[Theorem 7.2.2.1]{dundas2012local}) applied to $k[G] \xrightarrow{} k[G]/J$. For the second, note that by the proof of Lemma \ref{artin-wed} below, $k[G]/J$ is Morita equivalent to a product of perfect fields, so that $\tc_i(k[G]/J;p)=0$ for $i>0$ by \cite[Theorem B]{hesselholt1997k}.
\end{proof}
\begin{remark}
    By \cite[Theorem A]{geisser2011relative}, $K(k[G],J) \simeq \tc(k[G],J)$ has $p$-torsion homotopy groups that are of bounded exponent in each degree. In particular, these spectra are $p$-complete.
\end{remark}

Combining the previous lemma and proposition, we see that $K_i(k[G])\cong \tc_i(k[G];p) \oplus K_i(k[G]/J)$ for $i>0$, and so we deduce from Theorem \ref{main-thm} the stated Theorem \ref{intro-thm-a} of the introduction. More generally, we observe that the $K$-theory of finite group rings over a perfect field of positive characteristic is entirely determined by $\tc$ and the $K$-theory of the semisimple quotient. The latter is well understood, as we now explain.

\begin{lemma}\label{artin-wed}
    Let $k$ be a perfect field of characteristic $p$ and $G$ a finite group. There is an equivalence
    \[ K(k[G]/J) \simeq \bigoplus_{j=1}^m K(k_j) \]
    for perfect fields $k_j$ (finite extensions of $k$).
\end{lemma}
\begin{proof}
     The semisimple algebra $k[G]/J$ splits as a product of matrix algebras over division algebras by the Artin-Wedderburn theorem, where each division algebra is the endomorphism ring of a simple representation. In positive characteristic, each such division algebra is in fact a field, by \cite[Exercise 9.7]{isaacs1994character}. It follows that $k[G]/J$ is Morita equivalent to a product of perfect fields.
\end{proof}

From Lemma \ref{artin-wed}, we may observe that $K_i(k[G]/J)$ is uniquely $p$-divisible for $i>0$ (in particular, $K(k[G]/J)^\wedge_p$ is concentrated in degree $0$). Indeed, this is the case for perfect $\mathbb{F}_p$-algebras by \cite[Theorem 5.4]{hiller1981lambda}.
\appendix

\section{Induction and Spectral Mackey Functors}

We study induction theory for invariants such as $K$-theory, $\tc$, and $\thh$ applied to group rings. By giving these invariants the structure of $G$-spectra, one can apply the higher categorical analogue of the usual induction and restriction theory for Mackey functors. We recall the construction, and show how projectivity of these spectral Mackey functors relates to projectivity of the usual $G_0$ Mackey functor. The induction results give little new information for $p$-groups, but are of great use in dealing with groups of composite order. There is some overlap between our account and \cite{vogeli2025derived}; see there for another perspective and many examples. Much of this material will be known to experts, and we make no claims of originality in the constructions described.

\subsection{Spectral Mackey functors and $G$-spectra}
We begin with a brief reminder of the Mackey functor approach to stable equivariant homotopy. Spectral Mackey functors provide a model for $G$-spectra which is particularly convenient in the context of algebraic $K$-theory, and we will recall how $K(R[G])$ and $\tc(R[G])$ may be promoted to $G$-spectra. For background on these topics, see \cite{barwick2017spectral}, \cite{barwick2019spectral}, \cite{mathew2017nilpotence}. 

For a finite group $G$, we denote by $\mathrm{Span}(\mathrm{Fin}_G)$ the category of spans of finite $G$-sets. In classical equivariant algebra, Mackey functors valued in a semi-additive category $\mathcal{C}$ may be defined as functors from $\mathrm{Span}(\mathrm{Fin}_G)$ into $\mathcal{C}$ that preserve direct sums (\cite{dress1973contributions}). To extend this approach to the setting of $\infty$-categories, one must take into account the full structure of $\mathrm{Span}(\mathrm{Fin}_G)$ as a $(2,1)$-category, with higher morphisms given by equivalences of spans. We will not dwell on the precise definition of $\mathrm{Span}(\mathrm{Fin}_G)$ as a higher category, see \cite[Definition 3.6]{barwick2017spectral}. Mackey functors with values in a semi-additive $\infty$-category $\mathcal{C}$ may then be defined as functors of $\infty$-categories out of $\mathrm{Span}(\mathrm{Fin}_G)$ (\cite[Definition 6.1]{barwick2017spectral}).

\begin{definition}\label{mack-fun}
    For a semi-additive category $\mathcal{C}$, the category of $\mathcal{C}$-valued Mackey functors is the category of direct sum preserving functors $\mathrm{Span}(\mathrm{Fin}_G) \xrightarrow{} \mathcal{C}$,
    \[ \mathrm{Mack}_G(\mathcal{C}) = \mathrm{Fun}^\oplus(\mathrm{Span}(\mathrm{Fin}_G), \mathcal{C}). \]
    We will also refer to $\mathrm{Mack}_G(\Sp)$ as the category of $G$-spectra.
\end{definition}

For $M\in \mathrm{Mack}_G(\mathcal{C})$, we write $M^H$ for $M(G/H)$ and refer to it as the $H$-fixed points of $M$, just as for usual Mackey functors. The classical presentation of Mackey functors in terms of fixed points and restriction, transfer and conjugation maps does not have an exact analogue in the $\infty$-categorical setting, due to the difficulty of specifying coherences, but is nonetheless a useful guide.

The usage of the name $G$-spectra to refer to $\mathrm{Mack}_G(\Sp)$ is justified by \cite{guillou2024models}. Definition \ref{mack-fun} may be applied to $\mathcal{C}=\Mod(R)$ for any ring spectrum $R$, and more generally to any stable category $\mathcal{C}$. However, allowing $\mathcal{C}$ to be merely semi-additive (as opposed to stable) is very useful: categories of categories tend to be semi-additive but not stable, so Definition \ref{mack-fun} can be used to talk about categorical Mackey functors.

The category $\mathrm{Span}(\mathrm{Fin}_G)$ carries a symmetric monoidal structure (\cite[Theorem 2.15]{barwick2019spectral}) via the cartesian product on $\mathrm{Fin}_G$. If $\mathcal{C}$ is presentably symmetric monoidal, then $\mathrm{Mack}_G(\mathcal{C})$ can be equipped with a presentably symmetric monoidal structure via Day convolution (see \cite[Section 3]{barwick2019spectral}, note that one must check that Day convolution is compatible with the localisation $\mathrm{Fun}(\mathrm{Span}(\mathrm{Fin}_G), \mathcal{C}) \xrightarrow{} \mathrm{Mack}_G(\mathcal{C})$). Commutative algebras for the Day convolution symmetric monoidal structure are also known as Green functors. An alternative point of view is that such commutative algebras are equivalent to lax symmetric monoidal functors out of $\mathrm{Span}(\mathrm{Fin}_G)$.

\begin{prop}
    There is an equivalence between $\mathrm{CAlg}(\mathrm{Mack}_G(\mathcal{C}))$ and the subcategory of direct sum preserving, lax symmetric monoidal functors $\mathrm{Span}(\mathrm{Fin}_G) \xrightarrow{} \mathcal{C}$.
\end{prop}
\begin{proof}
    This is a general property of Day convolution, see \cite[Proposition 2.12]{glasman2017day}.
\end{proof}

The category of $\mathcal{C}$-valued Mackey functors is related to the category of objects in $\mathcal{C}$ with a $G$-action, $\mathrm{Fun}(BG, \mathcal{C})$. Indeed, restriction along $BG \xrightarrow{} \mathrm{Span}(\mathrm{Fin}_G)$ produces a functor
\[ \mathrm{Mack}_G(\mathcal{C}) \xrightarrow{} \mathrm{Fun}(BG, \mathcal{C}) \]
that sends a Mackey functor to the underlying object with $G$-action. It has both a fully faithful left adjoint and a fully faithful right adjoint, given by left and right Kan extension respectively. By \cite[Section 8]{barwick2019spectral}, the above functor is symmetric monoidal, so that the right adjoint admits a lax symmetric monoidal structure.

\begin{definition}
    Let the Borelification functor
    \[ (-)_\mathrm{Bor}: \mathrm{Fun}(BG, \mathcal{C}) \xrightarrow{} \mathrm{Mack}_G(\mathcal{C})\]
    be the right adjoint to the restriction. Dually, let the coBorelification functor
    \[ (-)_\mathrm{coBor}: \mathrm{Fun}(BG, \mathcal{C}) \xrightarrow{} \mathrm{Mack}_G(\mathcal{C}) \]
    be the left adjoint.
\end{definition}

For $X \in \mathrm{Fun}(BG, \mathcal{C})$, there is a natural comparison map
\[ X_\mathrm{coBor} \xrightarrow{} X_\mathrm{Bor} \]
which on $H$ fixed points exhibits the norm map
\[ X_{hH} \xrightarrow{} X^{hH}. \]

\begin{prop}\label{cobor-mod-bor}
    Suppose $\mathcal{C}$ is a symmetric monoidal, semi-additive category. For $X \in \mathrm{CAlg}(\mathrm{Fun}(BG, \mathcal{C}))$, $X_\mathrm{coBor}$ has the structure of a $X_\mathrm{Bor}$-module and the comparison map
    \[ X_\mathrm{coBor} \xrightarrow{} X_\mathrm{Bor} \]
    is a map of $X_\mathrm{Bor}$-modules.
\end{prop}
\begin{proof}
    There is an equivalence
    \[ X_{\mathrm{coBor}} \simeq EG_+ \otimes X_\mathrm{Bor} \]
    where we use that $\mathrm{Mack}_G(\mathcal{C})$ is tensored over the category of $G$-spaces. Hence $X_{\mathrm{coBor}}$ has the structure of a $X_\mathrm{Bor}$-module. Moreover, the comparison map $ X_\mathrm{coBor} \xrightarrow{} X_\mathrm{Bor}$ is simply the tensor of $X_\mathrm{Bor}$ with the canonical map $EG_+ \xrightarrow{} S^0$, so is a map of $X_\mathrm{Bor}$-modules.
\end{proof}

\subsection{$G$-spectra from categorical invariants}
We now explain how to obtain $G$-spectra from categorical invariants such as algebraic $K$-theory and $\thh$. The most straightforward approach is first to construct Mackey functors with values in a category of categories, then to postcompose with a categorical invariant to decategorify and obtain $G$-spectra. We will explain how such techniques can be applied in the simple case of studying representations over a ring $R$. For a more general construction, see \cite[Section 2]{clausen2020descent}.

\begin{definition}
    Let $\mathrm{Cat}^\mathrm{perf}$ denote the category of small, stable, idempotent complete categories and exact functors between them.
\end{definition}

Suppose $R$ is a commutative ring. We may view $\mathrm{Mod}(R)^\omega$ as an object in $\mathrm{Fun}(BG, \mathrm{Cat}^\mathrm{perf})$ by equipping it with the trivial $G$-action. We then obtain the Borelification (denoted $\mathrm{Mod}(R)^\omega_\mathrm{Bor}$) and the coBorelification (denoted $\mathrm{Mod}(R)^\omega_\mathrm{coBor}$). By Proposition \ref{cobor-mod-bor}, $\mathrm{Mod}(R)^\omega_\mathrm{Bor}$ is a commutative algebra in $\mathrm{Mack}_G(\mathrm{Cat}^\mathrm{perf})$ and $\mathrm{Mod}(R)^\omega_\mathrm{coBor}$ is a $\mathrm{Mod}(R)^\omega_\mathrm{Bor}$-module.

The fixed points of $\mathrm{Mod}(R)^\omega_\mathrm{Bor}$ have a very simple description. Indeed, by \cite[Theorem 1.1.4.4]{lurie2012} limits in $\mathrm{Cat}^\mathrm{perf}$ are computed as the limit in $\mathrm{Cat}$, so we see that
\[ (\mathrm{Mod}(R)^\omega_\mathrm{Bor})^H\simeq (\mathrm{Mod}(R)^\omega)^{hH}\simeq \mathrm{Fun}(BH, \mathrm{Mod}(R)^\omega). \]
We write $\mathrm{Rep}_R(H)$ for $\mathrm{Fun}(BH, \mathrm{Mod}(R)^\omega)$, and refer to it as the category of $H$ representations over $R$. The category $\mathrm{Rep}_R(H)$ is symmetric monoidal, with tensor unit the trivial representation.  Note that if $R$ is a field, then $\mathrm{Rep}_R(H)$ is simply the bounded derived $\infty$-category of finite dimensional $R[H]$-modules. It is not hard to see that the restriction and transfer maps in the Mackey functor $\mathrm{Rep}_R(-)=\mathrm{Mod}(R)^\omega_\mathrm{Bor}$ arise from the usual restriction and induction of representations. 

Similarly, we can describe the fixed points of $\mathrm{Mod}(R)^\omega_\mathrm{coBor}$. By \cite[Example 2.17]{clausen2020descent}, for $\mathcal{C} \in \mathrm{Fun}(BG,\mathrm{Cat}^\mathrm{perf})$, the homotopy orbits $\mathcal{C}_{hG}$ are described as the full subcategory of compact objects in $\mathrm{Ind}(\mathcal{C})^{hG}$, so that 
\[ (\mathrm{Mod}(R)^\omega_\mathrm{coBor})^H\simeq (\mathrm{Mod}(R)^\omega)_{hH}\simeq \mathrm{Fun}(BH, \mathrm{Mod}(R))^\omega\simeq \mathrm{Mod}(R[H])^\omega. \]
Again, the structure of the Mackey functor comes from restriction and induction of $R[H]$-modules. Note that although $\mathrm{Mod}(R[H])$ is symmetric monoidal, $\mathrm{Mod}(R[H])^\omega$ does not in general inherit a unital monoidal structure, since the unit of $\mathrm{Mod}(R[H])$ is not compact when $|H|$ is not invertible in $R$.

By \cite{clausen2020descent}, the norm map
\[ (\mathrm{Mod}(R)^\omega_\mathrm{coBor})^H\simeq \mathrm{Fun}(BH, \mathrm{Mod}(R))^\omega \xrightarrow{} \mathrm{Fun}(BH, \mathrm{Mod}(R)^\omega) \simeq (\mathrm{Mod}(R)^\omega_\mathrm{Bor})^H \]
agrees with the fully faithful inclusion $\mathrm{Fun}(BH, \mathrm{Mod}(R))^\omega \subseteq \mathrm{Fun}(BH, \mathrm{Mod}(R)^\omega)$. From this perspective, the fact that $\mathrm{Mod}(R)^\omega_\mathrm{coBor}$ is a $\mathrm{Mod}(R)^\omega_\mathrm{Bor}$-module witnesses the fact that $\mathrm{Fun}(BH, \mathrm{Mod}(R))^\omega$ is a tensor ideal of $\mathrm{Fun}(BH, \mathrm{Mod}(R)^\omega)$.

For ease of notation, we make the following definition.
\begin{definition}
    For $R \in \mathrm{CAlg}(\mathrm{Sp})$ and $E:\mathrm{Cat}^\mathrm{perf} \xrightarrow{} \mathcal{C}$ a product preserving functor, let $E(R[-])$ denote the $\mathcal{C}$-valued Mackey functor
    \[ E(R[-]) = E(\mathrm{Mod}(R)^\omega_\mathrm{coBor}).\]
    Let $E(\mathrm{Rep}_R(-))$ denote the $\mathcal{C}$-valued Mackey functor
    \[ E(\mathrm{Rep}_R(-)) = E(\mathrm{Mod}(R)^\omega_\mathrm{Bor}).\]
\end{definition}

\begin{example}
    Consider $\thh$ as a functor $\mathrm{Cat}^\mathrm{perf} \xrightarrow{} \mathrm{CycSp}$. Then we obtain two Mackey functors
    \[ \thh(R[-]), \thh(\mathrm{Rep}_R(-)) \]
    valued in cyclotomic spectra.
\end{example}

\subsection{Induction via $G$-spectra}
For an invariant $E$, we would like to determine the defect base of the Mackey functors $E_n(R[-])$ (in the sense of \cite{green1971axiomatic}). An upper bound for the defect base of $E_n(R[-])$ can be given by determining the defect base of the spectral Mackey functor $E(R[-])$. Moreover, when $E$ is lax symmetric monoidal, we have seen that $E(R[-])$ is a module over the spectral Green functor $E(\mathrm{Rep}_R(-))$, so in fact an upper bound is given by the defect base of $E(\mathrm{Rep}_R(-))$. When $E$ receives a trace map from $K$-theory, the defect base of $E(\mathrm{Rep}_R(-))$ is related to the defect base of $K(\mathrm{Rep}_R(-))$, which is also referred to as the $G$-theory of $R[G]$.

\begin{definition}
    For $\mathcal{F}$ a family of subgroups, we say that $R$-based $G$-theory induction holds for $\mathcal{F}$ if the Green functor $K_0(\mathrm{Rep}_R(-))$ is $\mathcal{F}$-projective.
\end{definition}

\begin{prop}\label{E-inv-defect}
    Let $E:\mathrm{Cat}^\mathrm{perf}\xrightarrow{} \mathrm{Sp}$ be a lax symmetric monoidal, product preserving functor, and suppose that there is a map $K \xrightarrow{} E$ of lax symmetric monoidal functors from algebraic $K$-theory. For $R \in \mathrm{CAlg}(\Sp)$ and $\mathcal{F}$ a family of subgroups such that $R$-based $G$-theory induction holds for $\mathcal{F}$, the Mackey functor $E(R[-])$ is $\mathcal{F}$-projective.
\end{prop}
\begin{proof}
    By assumption, the Mackey functor $K_0(\mathrm{Rep}_R(-))$ is $\mathcal{F}$-projective,
    which implies that the spectral Mackey functor $K(\mathrm{Rep}_R(-))$ is also $\mathcal{F}$-projective by \cite[Proposition 4.4]{mathew2019derived}. The existence of the map $K \xrightarrow{} E$ implies that $E(\mathrm{Rep}_R(-))$ is an algebra in $G$-spectra over $K(\mathrm{Rep}_R(-))$, and $E(R[-])$ is a module over $E(\mathrm{Rep}_R(-))$ by Proposition \ref{cobor-mod-bor}. Hence $E(R[-])$ is a module over $K(\mathrm{Rep}_R(-))$, so is also $\mathcal{F}$-projective.
\end{proof}

Proposition \ref{E-inv-defect} applies to algebraic $K$-theory itself, as well as the trace invariants ($\tc$, $\thh$ etc) that show up in the study of $K$-theory. When $R$ is a field (or more generally an algebra over a field), then Brauer's induction theorem and its relatives describe the defect base of $R$-based $G$-theory. Recall that a group is hyperelementary if it is a semi-direct product $C_n\rtimes Q$, where $Q$ is an $\ell$-group for some prime $\ell$ and $C_n$ is the cyclic group of order $n$, where $\ell \nmid n$. A hyperelementary group is elementary if it is a direct product $C_n \times Q$.

\begin{cor}\label{trace-inv-defect}
    Let $R$ be an $\mathbb{F}_p$-algebra. Then the Mackey functors
    \[ K_n(R[-]), \tc_n(R[-]), \mathrm{TR}_n(R[-]), \thh_n(R[-]) \]
    are all $\mathcal{F}$-projective for $\mathcal{F}$ the family of hyperelementary subgroups.

    Let $e$ be the $p'$-part of the exponent of $G$. If $R$ is an algebra over $\mathbb{F}_q$, where $\mathbb{F}_q$ contains a primitive $e^\mathrm{th}$ root of unity, then the above are all $\mathcal{F}$-projective for $\mathcal{F}$ the family of elementary subgroups.
\end{cor}
\begin{proof}
    Note that the property of being $\mathcal{F}$-projective is preserved under any product preserving functor, so for $E$ an $\mathcal{F}$-projective $G$-spectrum, the homotopy groups $\pi_n(E)$ are $\mathcal{F}$-projective Mackey functors. Hence it suffices to show that the relevant $G$-spectra are $\mathcal{F}$-projective. The functors $\tc$, $\mathrm{TR}$, and $\thh$ all receive trace maps from $K$-theory (\cite{madsen1995algebraic}), so by Proposition \ref{E-inv-defect}, it only remains to show that the relevant $R$-based $G$-theory induction holds. In the first case, $K_0(\mathrm{Rep}_R(-))$ is a module over $K_0(\mathrm{Rep}_{\mathbb{F}_p}(-))$ via extension of scalars. The Green functor $K_0(\mathrm{Rep}_{\mathbb{F}_p}(-))$ is projective with respect to the class of hyperelementary subgroups, see \cite{dress1973contributions}.

    In the second case, $K_0(\mathrm{Rep}_R(-))$ is similarly a module over $K_0(\mathrm{Rep}_{\mathbb{F}_q}(-))$, so it suffices to show that the unit of $K_0(\mathrm{Rep}_{\mathbb{F}_q}(G))$ lies in the sum of the images of induction from elementary subgroups. Let $p^a$ be the $p$-part of the exponent of $G$, and let $K$ be the field obtained from $W(\mathbb{F}_q)[1/p]$ by adjoining a primitive $p^{a\,\mathrm{th}}$ root of unity. This is a complete discretely valued field of characteristic $0$ with residue field $\mathbb{F}_q$, and it contains the $\exp(G)^\mathrm{th}$ roots of unity (the $e^\mathrm{th}$ roots of unity lift to $W(\mathbb{F}_q)$ as Teichm\"uller representatives). By Brauer's induction theorem over such fields (\cite[\S 12.6]{serre1977linear}), the unit of $R_K(G)$ is a $\mathbb{Z}$-linear combination of classes induced from elementary subgroups. The decomposition map $R_K(-) \xrightarrow{} K_0(\mathrm{Rep}_{\mathbb{F}_q}(-))$ preserves units and commutes with induction (\cite[\S 15.2, \S 18.1]{serre1977linear}), so the same holds in $K_0(\mathrm{Rep}_{\mathbb{F}_q}(G))$.
\end{proof}

\begin{example}
    Let $R$ be an $\mathbb{F}_p$-algebra. By Corollary \ref{trace-inv-defect}, we have decompositions
    \begin{align*}
        \tc_i(R[G]) &\simeq \underset{\mathcal{O}_{\mathcal{F}}(G)}{\colim} \ \tc_i(R[H]) \\
        \thh_i(R[G]) &\simeq \underset{\mathcal{O}_{\mathcal{F}}(G)}{\colim} \ \thh_i(R[H])
    \end{align*}
    where $\mathcal{F}$ is the family of hyperelementary subgroups.
\end{example}

\subsection{Brown's complex}
We now show that Corollary \ref{trace-inv-defect} implies that $\tc_i(R[G];p)$ (and so the $p$-torsion in $K_i(R[G])$) is a $p$-local invariant for $i>0$, when $R$ is a perfect algebra over a sufficiently large finite field. By a $p$-local invariant, we mean an invariant that is determined by its values on subgroups that contain a normal $p$-subgroup. Via \cite[Theorem A]{webb1991split}, the theory of such invariants is related to Brown's simplicial complex of non-trivial $p$-subgroups.

\begin{definition}
    Let $\Delta_p(G)$ denote the $G$-simplicial complex given by the geometric realisation of the poset of non-trivial $p$-subgroups of $G$.
\end{definition}

The complex $\Delta_p(G)$ has been much studied (\cite{brown1975euler}, \cite{brown1976high}, \cite{quillen1978homotopy}, \cite{grodal2002higher}). A simplex in $\Delta_p(G)$ is a chain $\sigma = P_0 < P_1 < \ldots < P_n$ where the $P_i$ are non-trivial $p$-subgroups. The stabiliser of $\sigma$ is then given by
\[ G_\sigma=N(P_0 < P_1 < \ldots < P_n) = N(P_0)\cap N(P_1)\cap\ldots\cap N(P_n). \]
Any such normaliser contains the non-trivial $p$-subgroup $P_0$ as a normal subgroup, so that $O_p(G_\sigma)\neq 1$ for any simplex $\sigma$. The simplicial complex $\Delta_p(G)$ is thus built from $p$-local information in a certain sense.

\begin{prop}\label{p-local}
    Let $e$ be the $p'$-part of the exponent of $G$, and let $R$ be a discrete, perfect algebra over $\mathbb{F}_{p^n}$, where $\mathbb{F}_{p^n}$ contains a primitive $e^\mathrm{th}$ root of unity. For each $i>0$, there is a split exact sequence
    \begin{align*}
    0 \to \tc_i(R[G];p)^\wedge_p \to \bigoplus_{\sigma \in (\Delta_p(G))_0/G} \tc_i(R[G_\sigma];p)^\wedge_p 
    &\to \bigoplus_{\sigma \in (\Delta_p(G))_1/G} \tc_i(R[G_\sigma];p)^\wedge_p \\
    &\to \bigoplus_{\sigma \in (\Delta_p(G))_2/G} \tc_i(R[G_\sigma];p)^\wedge_p \to \dots
    \end{align*}    
    where $\Delta_p(G)_n$ denotes the set of $n$-simplices of $\Delta_p(G)$.

    In particular, if $G$ contains a strongly $p$-embedded subgroup $H$, then $\tc_i(R[G];p)^\wedge_p \simeq \tc_i(R[H];p)^\wedge_p$.
\end{prop}
\begin{proof}
    We apply \cite[Theorem A]{webb1991split} with the Mackey functor $\tc_i(R[-])$ and the complex $\Delta_p(G)$. We take $\mathcal{X}$ to be the class of elementary subgroups and $\mathcal{Y}$ the class of subgroups with order prime to $p$. By Corollary \ref{trace-inv-defect}, $\tc_i(R[-];p)^\wedge_p$ is $\mathcal{X}$-projective. For $H \in \mathcal{Y}$, the group algebra $\mathbb{F}_{p^n}[H]$ is Morita equivalent to a product of copies of $\mathbb{F}_{p^n}$ (as $\mathbb{F}_{p^n}$ is a splitting field for $H$), hence $R[H]$ is Morita equivalent to a product of copies of $R$. The groups $\tc_i(R)^\wedge_p$ are $0$ for $i>0$ (\cite[Example 6.8]{clausen2021}), so $\tc_i(R[H];p)^\wedge_p=0$. 
    
    It remains to show that for $H \in \mathcal{X}-\mathcal{Y}$, the space $\Delta_p(G)^H$ is contractible. But by assumption, $H \cong C_n \times Q$ and does not have order prime to $p$, which implies that $H\cong P\times H'$ for some non-trivial $p$-group $P$. By the proof of \cite[Proposition 4.1]{quillen1978homotopy}, $\Delta_p(G)^P$ admits a contraction (via $Q' \leq Q'P \geq P$), which is $H'$-equivariant since $H'$ centralises $P$. Hence $\Delta_p(G)^{P\times H'}\simeq *$ as required.
\end{proof}

Proposition \ref{p-local} allows us to reduce the calculation of $\tc_i(R[G];p)^\wedge_p$ to $\tc_i(R[H];p)^\wedge_p$ for subgroups $H$ which normalise a chain of $p$-subgroups. Such methods have been used extensively in group cohomology, particularly in the calculation of cohomology of simple groups (\cite{webb1987local}, \cite{adem1991geometry}). 

\begin{example}\label{ex-lie-type}
    Suppose that $G$ is a finite group of Lie type in characteristic $p$.  The proof of Proposition \ref{p-local} still holds if we replace $\Delta_p(G)$ by a $G$-homotopy equivalent simplicial complex, and by \cite[Theorem 3.1]{quillen1978homotopy} $\Delta_p(G)$ is equivalent to the building of $G$.
\end{example}

\bibliography{refs}
\bibliographystyle{alpha}

\end{document}